\documentclass[11pt]{article}
\usepackage[tbtags]{amsmath}
\usepackage{amssymb}
\usepackage{amsthm}
\usepackage[misc]{ifsym}
\usepackage{cases}
\usepackage{cite}
\usepackage{mathrsfs}
\newcommand{\essinf}{\mathop{\mathrm{essinf}}}
\numberwithin{equation}{section}   
\normalsize
\title{\bf Stochastic Recursive Optimal Control Problem with Mixed Delay under Viscosity Solution's Framework
  \thanks{This work is financially supported by the National Key R\&D Program of China (2018YFB1305400), and the National Natural Science Foundations of China (11971266, 11571205, 11831010).}}
\author{\normalsize Weijun Meng\thanks{\it School of Mathematics, Shandong University, Jinan 250100, P.R. China, E-mail: 201611337@mail.sdu.edu.cn}, Jingtao Shi\thanks{\it Corresponding author, School of Mathematics, Shandong University, Jinan 250100, P.R. China, E-mail: shijingtao@sdu.edu.cn}}
\newtheorem{mypro}{Proposition}[section]
\newtheorem{mythm}{Theorem}[section]
\newtheorem{mydef}{Definition}[section]
\newtheorem{mylem}{Lemma}[section]

\newtheorem{Corollary}{Corollary}[section]
\begin{document}

\maketitle \noindent{\bf Abstract:}\quad This paper is concerned with the stochastic recursive optimal control problem with mixed delay. The connection between Pontryagin's maximum principle and Bellman's dynamic programming principle is discussed. Without containing any derivatives of the value function, relations among the adjoint processes and the value function are investigated by employing the notions of super- and sub-jets introduced in defining the viscosity solutions. Stochastic verification theorem is also given to
verify whether a given admissible control is really optimal.

\vspace{2mm}

\noindent{\bf Keywords:}\quad Stochastic recursive optimal control, mixed delay, maximum principle, dynamic programming principle, viscosity solution

\vspace{2mm}

\noindent{\bf Mathematics Subject Classification:}\quad 93E20, 60H10, 34K50, 91G80

\section{Introduction}

The stochastic optimal control problem has been an important research topic recently. Many papers have been published in the past few years, for example \cite{Kus72,Bis78,Ben81,Peng90,Zhou91,Peng93,Xu95,Wu98,YZ99,SW06,Yong10,Wu13,Hu17,HJX18}. However, in the above literatures, controlled dynamic systems only depend on the current state. In fact, the development of many dynamic systems not only depends on the state of the current time, but also depends on their previous history. Generally, stochastic functional differential equations (SFDEs) are used to describe these systems. The detailed study about SFDEs can be referred to Mohammed \cite{Moh84,Moh98}. A class of special SFDEs, called stochastic delayed differential equations (SDDEs), are usually studied recently. Due to its wide applications in engineering, life science and finance (see \cite{OS2000,GM06,AHMP07,KSW07,Fed11}, etc.), the SDDE has become a hot issue in modern research.

As the classic stochastic optimal control problems, the dynamic programming principle can be used to study the stochastic control problems with delay. However, because the system with delay is essentially defined in infinite dimensional space, this kind of research method will be very difficult to implement. Generally, people use three methods to apply the dynamic programming principle to the delayed systems. One approach is to simplify the primitive infinite dimensional system to a finite dimensional one. For example, Larssen \cite{Lar02}, Larssen and Risebro \cite{LR03}) studied the optimal control problem of SFDEs with bounded memory, and proved that the dynamic programming principle is still valid under this framework. They also proved that the value function is a smooth solution to the HJB equation. Another approach is to convert the equation with delay into that without delay. Bauer and Rieder \cite{BR05} studied a control system for SDDE in which the average of current and past values in a sense affects the development of the state. Through a special transformation, the infinite dimensional problem with delay can be transformed into a finite dimensional control problem without delay, and then the value function of the original problem is obtained by solving the simplified problem. The third method is to introduce the infinite generators. Chen and Wu \cite{CW12} studied a class of recursive optimal control problems for delayed systems characterized by SFDEs, and proved that the value function still satisfies the dynamic programming principle. Furthermore, by introducing the weak infinitesimal generators, the joint quadratic variation is used to obtain an infinite dimensional HJB equation, and it is further proved that the value function is the viscosity solution to this infinite dimensional partial differential equation.

Another way to solve the stochastic optimal control problems with delay is the maximum principle. A pioneer work by Hu and peng \cite{HP96}, obtained the maximum principle of a functional type stochastic system by introducing adjoint equations under the condition that the control domain is convex.  At present, the research on this topic can be divided into two directions. One direction involves three coupled adjoint equations, usually two backward stochastic differential equations (BSDEs) and a backward ordinary differential equation (ODE). For example, \O ksendal and Sulem \cite{OS2000} studied a class of optimal control problems in which the wealth equation is a stochastic differential equation with mixed delay (SMDDE). In their model, the current value, the value at one time in the past, and the average of the past value in some sense all affect the increase in wealth at current time. Under this framework, they obtained the maximum principle and applied the results to the relevant problems in finance. In the above paper, one of the main assumptions is that the third adjoint equation exists with zero solution to the backward ODE. In fact, this assumption basically reduces the infinite dimension control problem to the finite dimension problem. Therefore, the optimality conditions obtained in this direction may only be maintained when the delayed system is essentially finite dimensional. In the second direction, the adjoint equation is given by the anticipated backward stochastic differential equation (ABSDE), whose general theory is established by Peng and Yang \cite{PY09}. For example, Chen and Wu \cite{CW10} considered a delayed control system, which considered that the development of the current state only depends on a limited number of points in the past, but both the state variable and the control variable contain delay terms, and the control and its delay terms can enter the diffusion term. Under the assumption that the control domain is a convex set, they obtained the maximum principle for the first time by using ABSDE. Subsequently, Chen and Wu \cite{CW09} studied a stochastic recursive optimal control problem with time delay, the maximum principle is obtained under the condition that the control domain is non-convex and the diffusion term contains no control variable.

As two important tools to study stochastic optimal control problems with delay, there should exist some internal relations between the dynamic programming principle and maximum principle. However, so far, the relevant references are scarce. To the best of our knowledge, Shi \cite{Shi11} first investigated the relationship between maximum principle and dynamic programming principle for one kind of stochastic control systems with mixed delay, under the assumption that the value function is smooth. Shi et al. \cite{SXZ15} studied the connection between adjoint variables and the value function for stochastic recursive optimal control problem with mixed delay, but they still supposed that the value function is smooth. For recent literatures about the relationship between maximum principle and dynamic programming principle for stochastic control problems, refer to \cite{Zhou90,Zhou91,YZ99,Shi10,SY13,NSW16,NSW17,HJX18-2,Shi18,MS18,Shi19} and the references therein.

In this paper, we fill in the gap of the results in \cite{SXZ15}, under the viscosity solution framework, without the illusive assumption that the value function is smooth. For detailed introduction to viscosity solutions to partial differential equations, refer to Crandall et al. \cite{CIL92}, Yong and Zhou \cite{YZ99}. For the relationship between maximum principle and dynamic programming principle with in the framework of viscosity solutions, refer to \cite{Zhou90,Zhou91,YZ99,NSW17,HJX18-2,MS18,Shi19}. The contribution and innovation of this paper can be summarized as follows. For the convenience of presentation, first we display a prior estimate of the solution to the SMDDE. We also give a comparison theorem of SMDDE. Next, the connection between the adjoint variables and the value function for stochastic recursive optimal control problem with mixed delay under the viscosity solution framework, was given, which is consistent with the early results in Theorem 3.2 of \cite{SXZ15} when the value function is smooth. Finally the verification theorem is obtained which can help to verify if an admissible control is the optimal control.

The rest of this paper is organized as follows. In Section 2, some preliminary results about SMDDEs and ABSDEs are introduced. Next in Section 3, the problem is formulated and the sufficient condition of the maximum principle and the HJB equation are displayed. Then the main theorem is given in Section 4. The connection obtained in Section 4 can help us to seek the optimal control, thus in Section 5 the verification theorem is proved to verify if an admissible control is really optimal. Finally we give some useful concluding remarks in Section 6.

\section{Preliminary Results}

In this section, we give some preliminary results about SMDDEs and BSDEs.

Let $T>0$ be fixed, suppose $(\Omega,\mathcal{F},\mathbb{P})$ is a complete probability space. For $s\in[0,T)$, we define the filtration $\{\mathcal{F}_t^s\}_{t\geq s}=\sigma\{W(r)-W(s);s \leq r\leq t\}$ where $\{W(t)\}_{t\geq 0}$ is a $d$-dimensional Brownian motion, $\mathbb{E}[\cdot]$ denotes the expectation under the probability measure $\mathbb{P}$, and $\mathbb{E}^{\mathcal{F}_t^s}[\cdot]\equiv\mathbb{E}[\cdot|\mathcal{F}_t^s]$ denotes the conditional expectation. If $s=0$, we write $\mathcal{F}_t^s=\mathcal{F}_t$.

We first introduce the following spaces which will be used later. For $s\in[0,T)$ and an integer $p$, we define
\begin{eqnarray*}\begin{aligned}
                            C([s,T];\mathbf{R}^n)&:=\Big\{\mathbf{R}^n\mbox{-valued continuous funciton }\phi(t);\ \sup\limits_{s\leq t\leq T}|\phi(t)|<\infty\Big\},\\
L^p(\Omega,\mathcal{F}_T,\mathbb{P};\mathbf{R}^n)&:=\Big\{\mathbf{R}^n\mbox{-valued }\mathcal{F}_T\mbox{-measurable random variable }\xi;\ \mathbb{E}[|\xi|^p]<\infty\Big\},\\
          L^{2,p}_\mathcal{F}([s,T];\mathbf{R}^n)&:=\Big\{\mathbf{R}^n\mbox{-valued }\mathcal{F}_t^s\mbox{-adapted process }\phi(t);\
                                                  \mathbb{E}^{\mathcal{F}_t^s}\Big[\Big(\int_s^T|\phi(t)|^2dt\Big)^{\frac{p}{2}}\Big]<\infty\Big\},\\
              S^p_\mathcal{F}([s,T];\mathbf{R}^n)&:=\Big\{\mathbf{R}^n\mbox{-valued }\mathcal{F}_t^s\mbox{-adapted process }\phi(t);\ \mathbb{E}^{\mathcal{F}_t^s}\Big[\sup\limits_{s\leq t\leq T}|\phi(t)|^p\Big]<\infty\Big\}.
\end{aligned}\end{eqnarray*}

Let $\delta>0$ be fixed and $\lambda\in\mathbf{R}$ be a constant. For any $(s,\varphi)\in[0,T)\times C([-\delta,0];\mathbf{R}^n)$, we consider the following SMDDE:
\begin{eqnarray}\label{eq1}\left\{\begin{aligned}
 dX(t)&=b(t,X(t),X_1(t),X_2(t))dt+\sigma(t,X(t),X_1(t),X_2(t))dW(t),\ t\in[s,T],\\
  X(t)&=\varphi(t-s),\ t\in[s-\delta,s],
\end{aligned}\right.\end{eqnarray}
where
\begin{equation}\label{eq2}
  X_1(t):=\int_{-\delta}^0 e^{\lambda\tau}X(t+\tau)d\tau,\ X_2(t):=X(t-\delta),
\end{equation}
and $b:[s,T]\times\textbf{R}^n\times\textbf{R}^n\times\textbf{R}^n\rightarrow\textbf{R}^n$, $\sigma:[s,T]\times\textbf{R}^n\times\textbf{R}^n\times\textbf{R}^n\rightarrow\textbf{R}^{n\times d}$.

We impose the following assumptions on the coefficients of (\ref{eq1}).

\textbf{(A1)}
 (i) The functions $b,\sigma$ are globally Lipschitz with respect to $(x,x_1,x_2)$.

 (ii) There exists a constant $C$ such that for $\phi=b,\sigma$, the following holds:
\begin{equation*}
       |\phi(t,x,x_1,x_2)|\leq C(1+|x|+|x_1|+|x_2|),\ \forall x,x_1,x_2\in\mathbf{R}^n\times\mathbf{R}^n\times\mathbf{R}^n,\ t\geq 0.
\end{equation*}

 (iii) The functions $b,\sigma$ are measurable and $\varphi\in C([-\delta,0];\textbf{R}^n)$.

The following classical result can be found in Mohammed \cite{Moh98}.
\begin{mylem}\label{lem2.1}
  Under \textbf{(A1)}, the SMDDE (\ref{eq1}) admits a unique solution $X(\cdot)\in\mathcal{S}_{\mathcal{F}}^2([s,T];\mathbf{R}^n)$.
\end{mylem}

In the following an estimate of the solution to the SMDDE is given.
\begin{mylem}\label{lem2.2}
Suppose (\textbf{A1}) hold. Then for $p\geq 2$, the solution to SMDDE (\ref{eq1}) satisfies the following estimate:
\begin{equation}\begin{aligned}\label{eq3}
  \mathbb{E}^{\mathcal{F}_s^s}\bigg[\sup\limits_{s\leq t \leq T}|X(t)|^p\bigg]&\leq C(T,D,p,\delta,\lambda)\bigg[\sup\limits_{-\delta\leq r \leq 0}|\varphi(r)|^p+\bigg(\int_s^T|b(r,0,0,0)|dr\bigg)^p\\
                                                                              &\qquad+\bigg(\int_s^T|\sigma(r,0,0,0)|^2dr\bigg)^{\frac{p}{2}}\bigg],\ \mathbb{P}\mbox{-}a.s.,
\end{aligned}\end{equation}
where $D$ is the Lipschitz constant, and $C(T,D,p,\delta,\lambda)$ is a constant related to $T,D,p,\delta,\lambda$.
\end{mylem}
\begin{proof}
By \textbf{(A1)} and the Burkholder-Davis-Gundy inequality, we have
\begin{equation*}\begin{aligned}
       &\mathbb{E}^{\mathcal{F}_s^s}\bigg[\sup\limits_{s\leq t \leq T}|X(t)|^p\bigg]\\
  \leq &\ C(p)|X(s)|^p+C(p)\mathbb{E}^{\mathcal{F}_s^s}\bigg(\int_s^T|b(r,X(r),X_1(r),X_2(r))|dr\bigg)^p\\
       &+C(p)\mathbb{E}^{\mathcal{F}_s^s}\bigg(\int_s^T|\sigma(r,X(r),X_1(r),X_2(r))|^2dr\bigg)^{\frac{p}{2}}\\
  \leq &\ C(p)|\varphi(0)|^p+C(p,D)\mathbb{E}^{\mathcal{F}_s^s}\bigg[\bigg(\int_s^T\Big(|b(r,0,0,0)|+|X(r)|+|X_1(r)|+|X_2(r)|\Big)dr\bigg)^p\bigg]\\
       &+C(p,D)\mathbb{E}^{\mathcal{F}_s^s}\bigg[\bigg(\int_s^T\Big(|\sigma(r,0,0,0)|+|X(r)|+|X_1(r)|+|X_2(r)|\Big)^2dr\bigg)^{\frac{p}{2}}\bigg]\\
  \leq &\ C(p)|\varphi(0)|^p+C(p,D,T)\mathbb{E}^{\mathcal{F}_s^s}\bigg[\bigg(\int_s^T|b(r,0,0,0)|dr\bigg)^p+\bigg(\int_s^T|\sigma(r,0,0,0)|^2dr\bigg)^{\frac{p}{2}}\\
       &\quad+\int_s^T|X(r)|^pdr+\int_s^T|X_1(r)|^pdr+\int_s^T|X_2(r)|^pdr\bigg].
\end{aligned}\end{equation*}
Noting
\begin{equation*}\begin{aligned}
  &\mathbb{E}^{\mathcal{F}_s^s}\bigg[\int_s^T|X_2(r)|^pdr\bigg]=\mathbb{E}^{\mathcal{F}_s^s}\bigg[\int_s^T|X(r-\delta)|^pdr\bigg]=\mathbb{E}^{\mathcal{F}_s^s}\bigg[\int_{s-\delta}^{T-\delta}|X(r)|^pdr\bigg]\\
  &\leq \delta \bigg[\sup\limits_{-\delta\leq r \leq 0}|\varphi(r)|^p\bigg]+\mathbb{E}^{\mathcal{F}_s^s}\bigg[\int_s^T|X(r)|^pdr\bigg],
\end{aligned}\end{equation*}
and
\begin{equation*}\begin{aligned}
  &\mathbb{E}^{\mathcal{F}_s^s}\bigg[\int_s^T|X_1(r)|^pdr\bigg]=\mathbb{E}^{\mathcal{F}_s^s}\bigg[\int_s^T\bigg|\int_{-\delta}^0e^{\lambda\tau}X(r+\tau)d\tau\bigg|^pdr\bigg]\\
  &\leq \mathbb{E}^{\mathcal{F}_s^s}\bigg[\int_s^T\bigg(\int_{-\delta}^0e^{q\lambda\tau}d\tau\bigg)^{p-1}\int_{-\delta}^0|X(r+\tau)|^pd\tau dr\bigg]\\
  &\leq C(\delta,\lambda,p)\mathbb{E}^{\mathcal{F}_s^s}\bigg[\int_s^T\int_{-\delta}^0|X(r+\tau)|^pd\tau dr\bigg]\\
  &=C(\delta,\lambda,p)\mathbb{E}^{\mathcal{F}_s^s}\bigg[\int_{-\delta}^0\int_{s+\tau}^{T+\tau}|X(u)|^pdud\tau\bigg]\\
\end{aligned}\end{equation*}
\begin{equation*}\begin{aligned}  
  &\leq C(\delta,\lambda,p)\mathbb{E}^{\mathcal{F}_s^s}\bigg[\int_{-\delta}^0\int_{s+\tau}^{s}|X(u)|^pdud\tau \bigg]+C(\delta,\lambda,p)\mathbb{E}^{\mathcal{F}_s^s}\bigg[\int_{-\delta}^0\int_{s}^{T}|X(u)|^pdud\tau \bigg]\\
  &\leq C(\delta,\lambda,p)\int_{s-\delta}^s(u-s+\delta)|\varphi(u-s)|^pdu+C(\delta,\lambda,p)\mathbb{E}^{\mathcal{F}_s^s}\bigg[\int_s^T|X(r)|^pdr\bigg]\\
  &\leq C(\delta,\lambda,p)\Big[\sup\limits_{-\delta\leq r \leq 0}|\varphi(r)|^p\Big]+C(\delta,\lambda,p)\mathbb{E}^{\mathcal{F}_s^s}\bigg[\int_s^T|X(r)|^pdr\bigg],
\end{aligned}\end{equation*}
applying the Gronwall inequality, then we complete the proof of (\ref{eq3}).
\end{proof}

Next we give the comparison theorem of SMDDE.
\begin{mylem} \label{lem2.3}
Suppose $n=1$. For $(s,\varphi)\in[0,T)\times C([-\delta,0];\mathbf{R}^n)$, consider the following two SMDDEs, for $i=1,2$,
\begin{equation*}\left\{\begin{aligned}
  dX^i(t)&=b^i(t,X^i(t),X_2^i(t))dt+\sigma^i(t,X^i(t))dW(t),\ t\in[s,T],\\
   X^i(t)&=\varphi^i(t-s),\ t\in[s-\delta,s],
\end{aligned}\right.\end{equation*}
where $X_2^i(t):=X^i(t-\delta)$. Suppose $b^i,\sigma^i$ satisfy the assumption \textbf{(A1)}, and
\begin{equation*}\begin{aligned}
  &\varphi^1(t-s)\geq\varphi^2(t-s),\ \forall t\in[s-\delta,s],\\
  &b^1(t,X^2(t),X^2_2(t))\geq b^2(t,X^2(t),X^2_2(t)),\ a.e.\ t\in[s,T],\ a.s.\\
  &\sigma^1(t,X^2(t))=\sigma^2(t,X^2(t)),\ a.e.\ t\in[s,T],\ a.s.
\end{aligned}\end{equation*}
In addition, we assume $b^1(t,x,x_2)$ is increasing with respect to $x_2$, then $X^1(t)\geq X^2(t),\ a.e.\ t\in[s,T]$,\ a.s. Furthermore,
\begin{equation*}\begin{aligned}
  X^1(t)=X^2(t)\Longleftrightarrow&\varphi^1(t-s)=\varphi^2(t-s),\ t\in[s-\delta,s],\\ &b^1(t,X^2(t),X^2_2(t))=b^2(t,X^2(t),X^2_2(t)),\ a.e.\ t\in[s,T],\ a.s.
\end{aligned}\end{equation*}
\end{mylem}
\begin{proof}
Denote $\hat{X}:=X^1-X^2$, $\hat{X}_2:=X_2^1-X_2^2$, $\hat{\varphi}:=\varphi^1-\varphi^2$.
Then we have
\begin{equation*}\left\{\begin{aligned}
  d\hat{X}(t)&=[\alpha(t)\hat{X}(t)+\gamma(t)\hat{X}_2(t)+\hat{b}(t)]dt+\alpha'(t)\hat{X}(t)dW(t),\ t\in[s,T],\\
   \hat{X}(t)&=\hat{\varphi}(t-s),\ t\in[s-\delta,s],
\end{aligned}\right.\end{equation*}
where
\begin{equation*}\begin{aligned}
  \alpha(t):&=\frac{b^1(t,X^1(t),X^1_2(t))-b^1(t,X^2(t),X^1_2(t))}{\hat{X}(t)},\\
  \alpha'(t):&=\frac{\sigma^1(t,X^1(t))-\sigma^1(t,X^2(t))}{\hat{X}(t)},\\
  \gamma(t):&=\frac{b^1(t,X^2(t),X^1_2(t))-b^1(t,X^2(t),X^2_2(t))}{\hat{X}_2(t)},\\
  \hat{b}(t):&=b^1(t,X^2(t),X^2_2(t))-b^2(t,X^2(t),X^2_2(t)).
\end{aligned}\end{equation*}
In general the solution to SMDDE can be derived step by step. Let $t\in[s,s+\delta]$, then the above SMDDE becomes
\begin{equation*}\left\{\begin{aligned}
  d\hat{X}(t)&=[\alpha(t)\hat{X}(t)+\gamma(t)\hat{\varphi}(t-\delta-s)+\hat{b}(t)]dt+\alpha'(t)\hat{X}(t)dW(t),\\
   \hat{X}(s)&=\hat{\varphi}(0).
\end{aligned}\right.\end{equation*}
Apparently, it becomes a stochastic differential equation (SDE). Furthermore, by the comparison theorem of SDE, Since $\gamma(t)\geq 0$, $\hat{b}(t)\geq 0$, $\hat{\varphi}(t)\geq 0$, we deduce that $\hat{X}(t)\geq 0.$ Next we consider the case $t\in[s+\delta,s+2\delta]$, we can still get that $\hat{X}(t)\geq 0.$ Repeat the same steps on $[s+2\delta,s+3\delta]$, $[s+3\delta,s+4\delta],\cdots$ Since $T$ is finite, finally we can obtain $\hat{X}(t)\geq 0$ for $t\in[s,T]$.
\end{proof}

Now we consider the following backward stochastic differential equation (BSDE):
\begin{eqnarray}\label{eq4}\left\{\begin{aligned}
  -dY(t)&=f(t,Y(t),Z(t))dt-Z(t)dW(t),\ t\in[0,T],\\
    Y(T)&=\xi,
\end{aligned}\right.\end{eqnarray}
where $f:[0,T]\times\Omega\times\mathbf{R}^n\times\mathbf{R}^{n\times d}\rightarrow\mathbf{R}^n$ satisfies the following assumption.

\textbf{(A2)}: (i) For any $(y,z)\in\mathbf{R}^n\times\mathbf{R}^{n\times d}$, $f(\cdot,y,z)$ is $\mathcal{F}_t$-adapted and $\int_0^T|f(s,0,0)|ds\in L^2(\Omega,\mathcal{F}_T,\mathbb{P};\mathbf{R}^n)$.

(ii) There exists a constant $C>0$ such that
\begin{equation*}
  |f(t,y,z)-f(t,y',z')|\leq C(|y-y'|+|z-z'|),\ \forall y,y'\in\textbf{R}^n,\ z,z'\in\textbf{R}^{n\times d}.
\end{equation*}
The following result is classical, by Pardoux and Peng \cite{PP90}.
\begin{mylem}\label{lem2.4}
Under \textbf{(A2)}, for given $\xi\in L^2(\Omega,\mathcal{F}_T,\mathbb{P};\mathbf{R}^n)$, the BSDE (\ref{eq4}) admits a unique adapted solution pair $(Y(\cdot),Z(\cdot))\in \mathcal{S}_{\mathcal{F}}^2([0,T];\mathbf{R}^n)\times L_{\mathcal{F}}^{2}([0,T];\mathbf{R}^{n\times d})$.
\end{mylem}

\section{Problem Statement and the Viscosity Solution}

In this section, we state the problem and give some preliminary results.

Let $T>0$ be finite and $U\subset\mathbf{R}^k$ be a nonempty convex set. Given $s\in[0,T)$, we denote $\mathcal{U}^\omega[s,T]$ the set of all 5 tuples $(\Omega,\mathcal{F},\mathbb{P},W(\cdot);u(\cdot))$ satisfying the following conditions:\\
(i)\ $(\Omega,\mathcal{F},\mathbb{P})$ is a complete probability space;\\
(ii)\ ${\{W(t)\}}_{t\geq s}$ is a $d$-dimensional standard Brownian motion defined on $(\Omega,\mathcal{F},\mathbb{P})$ over $[s,T]$ (with $W(s)=0, a.s.$), and $\mathcal{F}^s_t=\sigma{\{W(r);s\leq r\leq t\}}$ augmented by all the $\mathbb{P}$-null sets in $\mathcal{F}$;\\
(iii)\ $u:[s,T]\times\Omega\rightarrow U$ is an $\{\mathcal{F}^s_t\}_{t\geq s}$-adapted process on $(\Omega,\mathcal{F},\mathbb{P})$.

We write $(\Omega,\mathcal{F},\mathbb{P},W(\cdot);u(\cdot))\in\mathcal{U}^\omega[s,T]$, but if there is no confusion, occasionally we will simply write $u(\cdot)\in \mathcal{U}^\omega[s,T]$.

For simplicity of representation, in the following of this paper we only consider the one dimension case. Let $\delta>0$ be fixed, for any $(s,\varphi)\in[0,T)\times C([-\delta,0];\mathbf{R})$, we consider the triple $(X^{s,\varphi,u}(\cdot),Y^{s,\varphi,u}(\cdot),Z^{s,\varphi,u}(\cdot))\in\mathbf{R}\times\mathbf{R}\times\mathbf{R}$ given by the following controlled coupled forward-backward stochastic differential equation with mixed delay (FBSMDDE):
\begin{eqnarray}\label{eq2.1}\left\{\begin{aligned}
 dX^{s,\varphi;u}(t)&=b(t,X^{s,\varphi;u}(t),X_1^{s,\varphi;u}(t),X_2^{s,\varphi;u}(t),u(t))dt\\
                    &\quad+\sigma(t,X^{s,\varphi;u}(t),X_1^{s,\varphi;u}(t),X_2^{s,\varphi;u}(t),u(t))dW(t),\ t\in[s,T],\\
-dY^{s,\varphi;u}(t)&=f(t,X^{s,\varphi;u}(t),X_1^{s,\varphi;u}(t),X_2^{s,\varphi;u}(t),Y^{s,\varphi;u}(t),Z^{s,\varphi;u}(t),u(t))dt\\
                    &\quad-Z^{s,\varphi;u}(t)dW(t),\ t\in[s,T],\\
  X^{s,\varphi;u}(t)&=\varphi(t-s),\ t\in[s-\delta,s],\\
  Y^{s,\varphi;u}(T)&=\phi(X^{s,\varphi;u}(T),X_1^{s,\varphi;u}(T)),
\end{aligned}\right.\end{eqnarray}
here
\begin{equation}\label{eq2.2}
  X_1^{s,\varphi;u}(t):=\int_{-\delta}^0 e^{\lambda\tau}X^{s,\varphi;u}(t+\tau)d\tau,\ X_2^{s,\varphi;u}(t):=X^{s,\varphi;u}(t-\delta),
\end{equation}
and $b:[s,T]\times\mathbf{R}\times\mathbf{R}\times\mathbf{R}\times\textbf{U}\rightarrow\mathbf{R},\ \sigma:[s,T]\times\mathbf{R}\times\mathbf{R}\times\mathbf{R}\times\textbf{U}\rightarrow\mathbf{R},\ f:[s,T]\times\mathbf{R}\times\mathbf{R}\times\mathbf{R}\times\mathbf{R}\times\mathbf{R}\times\textbf{U}\rightarrow\mathbf{R}$ and $\phi:\mathbf{R}\times\mathbf{R}\rightarrow\mathbf{R}$.

The cost functional is introduced as follows:
\begin{equation}\label{eq2.3}
  J(s,\varphi;u(\cdot)):=-Y^{s,\varphi;u}(s),\quad(s,\varphi)\in[0,T]\times C([-\delta,0];\mathbf{R}),
\end{equation}
and our stochastic recursive optimal control problem is the following.

\textbf{Problem\ (P).} For given $(s,\varphi)\in[0,T]\times C([-\delta,0];\mathbf{R})$, our object is to find $u^*(\cdot)\in\mathcal{U}^\omega[s,T]$, such that (\ref{eq2.1}) admits a unique solution and
\begin{equation}\label{eq2.4}
  J(s,\varphi;u^*(\cdot))=\underset{u(\cdot)\in\mathcal{U}^\omega[s,T]}{\essinf}J(s,\varphi;u(\cdot)).
\end{equation}

We define the value function
\begin{eqnarray}\label{eq2.5}\left\{\begin{aligned}
  V(s,\varphi)&=\underset{u(\cdot)\in\mathcal{U}^\omega[s,T]}{\essinf}J(s,\varphi;u(\cdot)),\quad(s,\varphi)\in[0,T]\times C([-\delta,0];\mathbf{R}), \\
  V(T,\varphi)&=-\phi(\varphi),\quad\varphi\in C([-\delta,0];\mathbf{R}).
\end{aligned}\right.\end{eqnarray}

Any $u^*(\cdot)\in \mathcal{U}^\omega[s,T]$ that achieves the above infimum is called an optimal control, and the corresponding solution triple $(X^*(\cdot),Y^*(\cdot),Z^*(\cdot))$ is called an optimal trajectory. We refer to $(X^*(\cdot),Y^*(\cdot),Z^*(\cdot),u^*(\cdot))$ as an optimal quadruple.

We impose the following assumptions on the coefficent of (\ref{eq2.1}):

\textbf{(H1)} (i) The functions $b=b(t,x,x_1,x_2,u), \sigma=\sigma(t,x,x_1,x_2,u)$ are continuously differentiable with respect to $(x,x_1,x_2)$ and their derivatives are bounded and continuous.

(ii) There exists a constant $C$ such that for $\phi=b,\sigma$,
\begin{equation*}
 |\phi(t,x,x_1,x_2,u)|\leq C(1+|x|+|x_1|+|x_2|),\ \forall x,x_1,x_2\in\mathbf{R},\ u\in \mathbf{R},\ t\geq 0.
\end{equation*}

(iii) The functions $b,\sigma$ are measurable, $\varphi:\Omega\rightarrow C([-\delta,0];\mathbf{R})$ is $\mathcal{F}_s^s$-measurable and \\$\mathbb{E}^{\mathcal{F}^s_s}\Big[\sup\limits_{-\delta\leq t\leq 0}|\varphi(t)|^4\Big]<\infty$.

\textbf{(H2)} (i) $f$ is measurable.

(ii) The functions $f=f(t,x,x_1,x_2,y,z,u), \phi=\phi(x,x_1)$ are twice continuously differentiable with respect to $(x,x_1,x_2,y,z)$ and their derivatives are bounded and continuous.

(iii) There exists a constant $C$ such that
\begin{equation*}
 |f(t,x,x_1,x_2,0,0,u)|+|\phi(x,x_1)|\leq C(1+|x|+|x_1|+|x_2|),\  \forall x,x_1,x_2\in \mathbf{R}^3,\ u\in \mathbf{R},\ t\geq 0.
\end{equation*}

By Lemma \ref{lem2.1} and Lemma \ref{lem2.4}, under \textbf{(H1)} and \textbf{(H2)}, for any admissible control $u(\cdot)$, FBSMDDE (\ref{eq2.1}) admits a unique adapted solution $(X^{s,\varphi;u}(\cdot),Y^{s,\varphi;u}(\cdot),Z^{s,\varphi;u}(\cdot))$ $\in\mathcal{S}_\mathcal{F}^4([s,T];\mathbf{R})\times\mathcal{S}_\mathcal{F}^4([s,T];\mathbf{R})\times L_\mathcal{F}^{2,2}([s,T];\mathbf{R})$.

The following lemma plays a crucial role in this paper, which belongs to Pardoux and Rascanu \cite{PR14}.
\begin{mylem}\label{lem3.1}
Let \textbf{(H1)} and \textbf{(H2)} hold. For given $(s,\varphi)\in[0,T)\times C([-\delta,0];\mathbf{R})$, $u(\cdot)\in\mathcal{U}^\omega[s,T]$, and $(X^{s,\varphi;u}(\cdot), Y^{s,\varphi;u}(\cdot),Z^{s,\varphi;u}(\cdot))\in\mathcal{S}^2_\mathcal{F}([s,T],\mathbf{R})\times \mathcal{S}^2_\mathcal{F}([s,T],\mathbf{R})\times L^2_\mathcal{F}([s,T];\mathbf{R})$ is the unique solution to (\ref{eq2.1}). Then the following estimate holds:
\begin{equation}\begin{aligned}\label{eq2.6}
  &\mathbb{E}^{\mathcal{F}^s_s}\Big[\underset{s\leq r\leq T}{\sup}|Y^{s,\varphi;u}(r)|^2\Big]
   \leq C\mathbb{E}^{\mathcal{F}^s_s}\Big[|\phi(X^{s,\varphi;u}(T),X_1^{s,\varphi;u}(T))|^2\\
  &\quad+\Big(\int_{s}^{T}f(r,X^{s,\varphi;u}(r),X_1^{s,\varphi;u}(r),X_2^{s,\varphi;u}(r),0,0,u(r))dr\Big)^2\Big].
\end{aligned}\end{equation}
\end{mylem}

Now we start to state the maximum principle (see Shi, Xu and Zhang \cite{SXZ15}). First, we denote $\vec{p}(\cdot)=(p_1(\cdot),p_2(\cdot),p_3(\cdot))^\top$, $\vec{q}(\cdot)=(q_1(\cdot),q_2(\cdot))^\top$, and define the Hamiltonian function $H:[0,T]\times\mathbf{R}\times\mathbf{R}\times\mathbf{R}\times\mathbf{R}\times\mathbf{R}\times\textbf{U}\times\mathbf{R}^3\times\mathbf{R}^2\times\mathbf{R}\rightarrow\mathbf{R}$ as:
\begin{equation}\begin{aligned}\label{eq2.7}
  &H(t,x,x_1,x_2,y,z,u,\vec{p},\vec{q},\gamma):=p_1b(t,x,x_1,x_2,u)\\
  &\quad+p_2(x-\lambda x_1-e^{-\lambda\delta}x_2)+q_1\sigma(t,x,x_1,x_2,u)-\gamma f(t,x,x_1,x_2,y,z,u).
\end{aligned}\end{equation}
Then for given $(s,\varphi)\in [0,T]\times C([-\delta,0];\mathbf{R})$, we introduce the following adjoint equations:
\begin{eqnarray}\label{eq2.8}\left\{\begin{aligned}
 d\gamma(t)&=-H_y^*(t)dt-H_z^*(t)dW(t),\ t\in[s,T],\\
  \gamma(s)&=1,
\end{aligned}\right.\end{eqnarray}
\begin{eqnarray}\label{eq2.9}\left\{\begin{aligned}
 -dp_1(t)&=H_x^*(t)dt-q_1(t)dW(t),\ t\in[s,T],\\
 -dp_2(t)&=H_{x1}^*(t)dt-q_2(t)dW(t),\ t\in[s,T],\\
 -dp_3(t)&=H_{x2}^*(t)dt,\ t\in[s,T],\\
   p_1(T)&=-\phi_x(x^*(T),x_1^*(T))\gamma(T),\\
   p_2(T)&=-\phi_{x1}(x^*(T),x_1^*(T))\gamma(T),\\
   p_3(T)&=0.\\
\end{aligned}\right.\end{eqnarray}
where for $\psi=b,\sigma,f,\phi$ and $\kappa=x,x_1,x_2,y,z$, we define
\begin{eqnarray*}\begin{aligned}
  \psi^*_\kappa(r):=\psi_\kappa(r,X^*(r),X_1^*(r),X_2^*(r),Y^*(r),Z^*(r),u^*(r)).
\end{aligned}\end{eqnarray*}

Under \textbf{(H1), (H2)}, SDE (\ref{eq2.8}) admits a unique solution $\gamma(\cdot)\in \mathcal{S}_{\mathcal{F}}^2([s,T];\mathbf{R})$.

\noindent\textbf{Remark 3.1.} To derive the sufficent condition of the maximum principle, we have to require that $p_3(t)\equiv 0$ for all $t\in[s,T]$. In this case, BSDE (\ref{eq2.9}) admits a unique solution $(\vec{p}(\cdot),\vec{q}(\cdot))\in \mathcal{S}_{\mathcal{F}}^2([s,T];\mathbf{R}^3)\times L_{\mathcal{F}}^2([s,T];\mathbf{R}^2)$.

The following theorem gives the sufficient condition of the maximum principle.
\begin{mythm}\label{thm2.1}
Assume $u^*(\cdot)\in\mathcal{U}^{\omega}[s,T]$, $(X^*(\cdot),Y^*(\cdot),Z^*(\cdot))$ is the corresponding state trajectory. Suppose $\gamma(\cdot),(\vec{p}(\cdot),\vec{q}(\cdot))$ are the solutions to (\ref{eq2.9}) and (\ref{eq2.10}), respectively. If
\begin{equation*}\begin{aligned}
(x,x_1,x_2,y,z,u)\mapsto&H(t,x,x_1,x_2,y,z,u,\vec{p},\vec{q},\gamma)\  is\ convex\ function,\ \forall t\in[s,T],\\
                        &\phi(x,x_1)=Mx+Nx_1,\ M,N\in\mathbf{R},\\
                        &p_3(t)\equiv0,\ \forall t\in[s,T],
\end{aligned}\end{equation*}
and
\begin{equation*}
  H_u^*(t)(u^*(t)-u)\leq 0,\ \forall u\in U,\ t\in[s,T],\ a.s.,
\end{equation*}
then $u^*(\cdot)$ is the optimal control.
\end{mythm}

At the end of this section, we present the generalized HJB equation which will be studied.
Since the value function is defined on the infinite dimensional space $[0,T]\times C([-\delta,0];\mathbf{R})$, the HJB equation which satisfies is also infinite dimensional. However, Let $V(s,\varphi)=V(s,x,x_1)$, in other words, the value function is independent of $x_2$, here $x=\varphi(0)$, $x_1=\int_{-\delta}^0e^{\lambda\tau}\varphi(\tau)d\tau$, $x_2=\varphi(-\delta)$, then the HJB equation becomes finite dimensional:
\begin{eqnarray}\label{eq2.10}\left\{\begin{aligned}
 -V_s(s,x,x_1)&+\underset{u\in\textbf{U}}\sup\ G\big(s,x,x_1,x_2,u,-V(s,x,x_1),-V_x(s,x,x_1),\\
              &\qquad\qquad-V_{xx}(s,x,x_1),-V_{x_1}(s,x,x_1)\big)=0,\quad(t,x,x_1)\in[0,T]\times\mathbf{R}^2,\\
    V(T,x,x_1)&=-\phi(x,x_1),\quad\emph{\emph{for\ all}}\ x,x_1\in\mathbf{R},
\end{aligned}\right.\end{eqnarray}
where the generalized Hamiltonian function $G:[0,T]\times\mathbf{R}\times\mathbf{R}\times\mathbf{R}\times\textbf{U}\times\mathbf{R}\times\mathbf{R}\times\mathbf{R}\times\mathbf{R}\rightarrow\mathbf{R}$ is defined as
\begin{equation}\label{eq2.11}\begin{split}
&G(s,x,x_1,x_2,u,k,p,R,q):=\frac{1}{2}R\sigma^2(s,x,x_1,x_2,u)+pb(s,x,x_1,x_2,u)\\
&\qquad+q(x-\lambda x_1-e^{-\lambda\delta }x_2)+f(s,x,x_1,x_2,k,p\sigma(s,x,x_1,x_2,u),u).\\
\end{split}\end{equation}

\noindent\textbf{Remark 3.2.} In fact, the coefficients of (\ref{eq2.10}) $b,\sigma,f$ all depend on $x_2$, thus we can not hope the solution to (\ref{eq2.10}) is independent of $x_2$. However, if we impose some conditions to $b,\sigma,f$, then the solution of (\ref{eq2.10}) can indeed not depend on $x_2$, see more details in Theorem 2.7 of \cite{SXZ15}. In this paper, we suppose that the solution to (\ref{eq2.10}) does not depend on $x_2$.

Next let us recall some notions under the framework of viscosity solutions. More details can be seen in the paragraphs by Crandall et al. \cite{CIL92} and Yong and Zhou \cite{YZ99}. First, we give the notation of the second-order super- and sub-jets. For $v\in C([0,T]\times\mathbf{R}^2)$ and $(t,\hat{x},\hat{x}')\in[0,T]\times\mathbf{R}^2$, we define
\begin{eqnarray*}\left\{\begin{aligned}
D_{t+,x,x'}^{1,2,1,+}v(t,\hat{x},\hat{x}')=&\Big\{(\Theta,p,q,P)\in\mathbf{R}\times\mathbf{R}\times\mathbf{R}\times\mathbf{R}\Big|v(s,x,x')\leq v(t,\hat{x},\hat{x}')
                                            +\langle\Theta,s-t\rangle\\
                                           &\quad+\langle p,x-\hat{x}\rangle+\frac{1}{2}\langle P(x-\hat{x}),x-\hat{x}\rangle+\langle q,x'-\hat{x}'\rangle+o(|s-t|+|x-\hat{x}|^2\\
                                           &\quad+|x'-\hat{x}'|^2),\mbox{ as\ }s\downarrow t,\ x\rightarrow\hat{x},\ x'\rightarrow\hat{x}'\Big\}, \\
D_{t+,x,x'}^{1,2,1,-}v(t,\hat{x},\hat{x}')=&\Big\{(\Theta,p,q,P)\in\mathbf{R}\times\mathbf{R}\times\mathbf{R}\times\mathbf{R}\Big|v(s,x,x')\geq v(t,\hat{x},\hat{x}')
                                            +\langle\Theta,s-t\rangle\\
                                           &\quad+\langle p,x-\hat{x}\rangle+\frac{1}{2}\langle P(x-\hat{x}),x-\hat{x}\rangle+\langle q,x'-\hat{x}'\rangle+o(|s-t|+|x-\hat{x}|^2\\
                                           &\quad+|x'-\hat{x}'|^2),\mbox{ as\ }s\downarrow t,\ x\rightarrow\hat{x},\ x'\rightarrow\hat{x}'\Big\}.
\end{aligned}\right.\end{eqnarray*}

Next, we present the definition of viscosity solution which will be used later. It is slightly different from the typical and most common used definition (see \cite{YZ99}). However, since $D_{t,x,x'}^{1,2,1,+}v(t,\hat{x},\hat{x}')\subseteq D_{t+,x,x'}^{1,2,1,+}v(t,\hat{x},\hat{x}')$ and $D_{t,x,x'}^{1,2,1,-}v(t,\hat{x},\hat{x}')\subseteq D_{t+,x,x'}^{1,2,1,-}v(t,\hat{x},\hat{x}')$, it's clear that the viscosity solutions defined by the one-side semijets are also the ones defined by the two-side semijets.
\begin{mydef}\label{def1}
A continuous function $v$ on $[0,T]\times\mathbf{R}^2$ is called a viscosity subsolution (resp., supersolution) to the HJB equation (\ref{eq2.10}), if $v(T,x,x_1)\leq(\geq)-\phi(x,x_1)$ and
\begin{equation}\label{eq2.12}\begin{aligned}
-\varphi_t(t_0,x_0,x'_0)+&\underset{u\in\textbf{U}}\sup\ G\big(t_0,x_0,x'_0,x_2,u,-\varphi(t_0,x_0,x'_0),-\varphi_x(t_0,x_0,x'_0),\\
&\qquad-\varphi_{xx}(t_0,x_0,x'_0),-\varphi_{x_1}(t_0,x_0,x'_0)\big)\leq(\geq)\ 0,\ \forall x_2\in\mathbf{R},
\end{aligned}\end{equation}
whenever $v-\varphi$ attains a local maximum (resp., minimum) at $(t_0,x_0,x'_0)$ in a right neighborhood of $(t_0,x_0,x'_0)$ for $\varphi\in C^{1,2,1}([0,T]\times\mathbf{R}^2).$ A function $v$ is called a viscosity solution to (\ref{eq2.10}) if it is both a viscosity subsolution and a viscosity supersolution to (\ref{eq2.10}).
\end{mydef}

Actually, the viscosity solution in Definition \ref{def1} is equivalent to the one defined by the second-order super- and sub-jets. Hence we present the following lemma.

\begin{mylem}\label{lem3.2} Let $v$ be an uppercontinuous function on $(0,T)\times\mathbf{R}^2$ and $(t_0,x_0,x'_0)\in (0,T)\times\mathbf{R}^2$. Then $(\Theta,p,q,P)\in D_{t+,x}^{1,2,1,+}v(t_0,x_0,x'_0)$ if and only if there exists a function $\phi\in C^{1,2,1}((0,T)\times\mathbf{R}^2)$ such that $v-\phi$ attains a strict global maximum at $(t_0,x_0,x'_0)\in(0,T)\times \mathbf{R}^2$ from the right hand side on $t_0$ and
\begin{equation*}\begin{aligned}
  &(\phi(t_0,x_0,x'_0),\phi_t(t_0,x_0,x'_0),\phi_x(t_0,x_0,x'_0),\phi_{x_1}(t_0,x_0,x'_0),\phi_{xx}(t_0,x_0,x'_0))\\
  &=(v(t_0,x_0,x'_0),\Theta,p,q,P).
\end{aligned}\end{equation*}
Moreover, if $v$ has polynomial growth, i.e., if
\begin{equation}\label{eq2.13}
  |v(t,x,x_1)|\leq C(1+|x|^k+|x_1|^k),\mbox{ for some }k\geq 1,\ (t,x,x_1)\in(0,T)\times\mathbf{R}^2,
\end{equation}
then $\phi$ can be chosen so that $\phi,\phi_t,\phi_x,\phi_{x_1},\phi_{xx}$ satisfy (\ref{eq2.13}).
\end{mylem}
The proof can be found in Yong and Zhou \cite{YZ99}, Zhou \cite{Zhou91}.

\section{Connection between the Adjoint Variables and the Value Function}

In this section, we investigate the connection between the adjoint variables and the value function. In the whole section, we suppose the value function is independent of $x_2$.

\begin{mythm}\label{th3.1}
Let \textbf{(H1), (H2)} hold. Let $(s,x,x_1)\in[0,T]\times\mathbf{R}\times\mathbf{R}$ be fixed. Suppose that $u^*(\cdot)$ is an optimal control, and the triple $(X^*(\cdot),Y^*(\cdot),Z^*(\cdot))$ is the corresponding optimal trajectory. $V\in C([0,T]\times\mathbf{R}\times\mathbf{R})$ is the value function. Let $\gamma(\cdot)\in\mathcal{S}^2_\mathcal{F}([s,T];\mathbf{R})$, $(\vec{p}(\cdot),\vec{q}(\cdot))\in\mathcal{S}_\mathcal{F}^2([s,T];\mathbf{R}^3)\times L^2_\mathcal{F}([s,T];\mathbf{R}^2)$ satisfy the adjoint equations (\ref{eq2.8}), (\ref{eq2.9}), respectively. Furthermore, assume that $p_3(t)\equiv 0$ for all $t\in[s,T]$. Then
\begin{equation}\label{eq3.1}\begin{aligned}
&\big\{p_1(t)(\gamma(t))^{-1}\big\}\subseteq D^{1,+}_{x}V(t,X^*(t),X_1^*(t)), \\
 &D^{1,-}_{x}V(t,X^*(t),X_1^*(t))\subseteq\big\{p_1(t)(\gamma(t))^{-1}\big\},\ \mbox{for all } t\in[s,T],\ \mathbb{P}\mbox{-}a.s.,
\end{aligned}\end{equation}
where $X_1^*(t):=\int_{-\delta}^0e^{\lambda\tau}X^*(t+\tau)d\tau$, $X_2^*(t):=X^*(t-\delta).$
\end{mythm}

\begin{proof}
We split the proof into several steps.

$Step\ 1$. Variation equation for SMDDE.

Fix a $t\in [s,T]$. For any $x'\in\mathbf{R}$, denote by $X^{t,x',X^*_1(t)}(\cdot)$ the solution to the following SMDDE on $[t,T]$:
\begin{eqnarray}\label{eq3.2}\left\{\begin{aligned}
  dX^{t,x',X^*_1(t)}(r)&=b\big(r,X^{t,x',X^*_1(t)}(r),X_1^{t,x',X^*_1(t)}(r),X_2^{t,x',X^*_1(t)}(r),u^*(r)\big)dr\\
                       &\quad+\sigma\big(r,X^{t,x',X^*_1(t)}(r),X_1^{t,x',X^*_1(t)}(r),X_2^{t,x',X^*_1(t)}(r),u^*(r)\big)dW(r),\\
   X^{t,x',X^*_1(t)}(t)&=x',\ X^{t,x',X^*_1(t)}(r)=X^*(r),\ r\in[t-\delta,t),
\end{aligned}\right.\end{eqnarray}
where
\begin{equation}\label{eq3.3}
   X_1^{t,x',X^*_1(t)}(r):=\int_{-\delta}^0e^{\lambda\tau} X^{t,x',X^*_1(t)}(r+\tau)d\tau,\ X_2^{t,x',X^*_1(t)}(r):= X^{t,x',X^*_1(t)}(r-\delta).
\end{equation}

It is clear that SMDDE (\ref{eq3.2}) can be regarded as an SMDDE on $(\Omega,\mathcal{F},\{\mathcal{F}^s_r\}_{r\geq s},\mathbb{P}(\cdot|\mathcal{F}^s_t)(\omega))$ for $\mathbb{P}$-$a.s.\ \omega$, where $\mathbb{P}(\cdot|\mathcal{F}^s_t)(\omega)$ is the regular conditional probability for given $\mathcal{F}^s_t$ defined on $(\Omega,\mathcal{F})$. For any $t\leq r\leq T$, set
\begin{equation*}\begin{aligned}
&\hat{X}(r):=X^{t,x',X^*_1(t)}(r)-X^*(r),\ \hat{X_1}(r):=X_1^{t,x',X^*_1(t)}(r)-X_1^*(r),\\
&\hat{X}_2(r):=X_2^{t,x',X^*_1(t)}(r)-X_2^*(r).
\end{aligned}\end{equation*}
Then we can write the equation for $\hat{X}(r)$ as follows:
\begin{eqnarray}\label{eq3.4}\left\{\begin{aligned}
  d\hat{X}(r)&=\big\{b^*_x(r)\hat{X}(r)+b_{x_1}^*(r)\hat{X_1}(r)+b_{x_2}^*(r)\hat{X_2}(r)+\varepsilon_1(r)\big\}dr\\
             &\quad+\big\{\sigma_x^*(r)\hat{X}(r)+\sigma_{x_1}^*(r)\hat{X_1}(r)+\sigma^*_{x_2}(r)\hat{X_2}(r)+\varepsilon_2(r)\big\}dW(r),\\
   \hat{X}(t)&=x'-X^*(t),\ \hat{X}(r)=0,\ r\in[t-\delta,t),
\end{aligned}\right.\end{eqnarray}
where
\begin{eqnarray}\left\{\begin{aligned}\label{eq3.5}
       \varepsilon_1(r)&:=[b_x^\theta(r)-b_x^*(r)]\hat{X}(r)+[b_{x_1}^\theta(r)-b_{x_1}^*(r)]\hat{X_1}(r)+[b_{x_2}^\theta(r)-b_{x_2}^*(r)]\hat{X_2}(r),\\
       \varepsilon_2(r)&:=[\sigma_x^\theta(r)-\sigma_x^*(r)]\hat{X}(r)+[\sigma_{x_1}^\theta(r)-\sigma_{x_1}^*(r)]\hat{X_1}(r)+[\sigma_{x_2}^\theta(r)-\sigma_{x_2}^*(r)]\hat{X_2}(r).
\end{aligned}\right.\end{eqnarray}
Here, for $\psi=b,\sigma$, and $\kappa=x,x_1,x_2$, we define
\begin{eqnarray}\label{eq3.6}\begin{aligned}
  \psi^\theta_\kappa(r):=\int_0^1\psi_\kappa\big(r,X^*(r)+\theta\hat{X}(r),X_1^*(r)+\theta\hat{X_1}(r),X_2^*(r)+\theta\hat{X_2}(r),u^*(r)\big)d\theta.
\end{aligned}\end{eqnarray}

Recall the estimate of the solution to SMDDE, by Lemma \ref{lem2.2} we obtain
\begin{equation}\begin{aligned}\label{eq3.7}
  \mathbb{E}^{\mathcal{F}_t^s}\Big[\sup\limits_{t\leq r\leq T}|\hat{X}(r)|^p\Big]\leq C|x'-X^*(t)|^p,\ \mathbb{P}\mbox{-}a.s.
\end{aligned}\end{equation}
Similarly, we have
\begin{equation}\begin{aligned}\label{eq3.8}
  &\mathbb{E}^{\mathcal{F}_t^s}\Big[\sup\limits_{t\leq r\leq T}|\hat{X_1}(r)|^p\Big]
   =\mathbb{E}^{\mathcal{F}_t^s}\Big[\sup\limits_{t\leq r\leq T}|\int_{-\delta}^0e^{\lambda\tau}\hat{X}(r+\tau)d\tau|^p\Big]\\
  &\leq \mathbb{E}^{\mathcal{F}_t^s}\Big[\sup\limits_{t\leq r\leq T}\Big|\int_{-\delta}^0e^{q\lambda\tau}d\tau\Big|^{p-1}\int_{-\delta}^0|\hat{X}(r+\tau)|^pd\tau\Big]\\
  &\leq C\mathbb{E}^{\mathcal{F}_t^s}\Big[\sup\limits_{t-\delta\leq r\leq T}|\hat{X}(r)|^p\Big]\leq C|x'-X^*(t)|^p,\ \mathbb{P}\mbox{-}a.s.,
\end{aligned}\end{equation}
\begin{equation}\begin{aligned}\label{eq3.9}
  \mathbb{E}^{\mathcal{F}_t^s}\Big[\sup\limits_{t\leq r\leq T}|\hat{X_2}(r)|^p\Big]
  \leq \mathbb{E}^{\mathcal{F}_t^s}\Big[\sup\limits_{t-\delta\leq r\leq T}|\hat{X}(r)|^p\Big]\leq C|x'-X^*(t)|^p,\ \mathbb{P}\mbox{-}a.s.
\end{aligned}\end{equation}

 $Step\ 2$. Estimate of remainder terms of SMDDE.

For any integer $p\geq2$, we have
\begin{eqnarray}\label{eq3.10}\left\{\begin{aligned}
   &\mathbb{E}^{\mathcal{F}_t^s}\bigg[\bigg(\int_t^T|\varepsilon_1(r)|^2dr\bigg)^{\frac{p}{2}}\bigg]= o(|x'-X^*(t)|^p),\ \mathbb{P}\mbox{-}a.s., \\
   &\mathbb{E}^{\mathcal{F}_t^s}\bigg[\bigg(\int_t^T|\varepsilon_2(r)|^2dr\bigg)^{\frac{p}{2}}\bigg]=o(|x'-X^*(t)|^p),\ \mathbb{P}\mbox{-}a.s.
\end{aligned}\right.\end{eqnarray}

To prove (\ref{eq3.10}), by the continuity of $b_x,b_{x_1},b_{x_2}$ as well as (\ref{eq3.7}), (\ref{eq3.8}) and (\ref{eq3.9}), we have
\begin{eqnarray}\nonumber\begin{aligned}
  &\mathbb{E}^{\mathcal{F}^s_t}\bigg[\bigg(\int_t^T|\varepsilon_1(r)|^2dr\bigg)^{\frac{p}{2}}\bigg]
   \leq C\mathbb{E}^{\mathcal{F}^s_t}\bigg[\bigg(\int_t^T[b_x^\theta(r)-b_x^*(r)]^2|\hat{X}(r)|^2 dr\bigg)^{\frac{p}{2}}\bigg]\\
  &\quad+C\mathbb{E}^{\mathcal{F}^s_t}\bigg[\bigg(\int_t^T[b_{x1}^\theta(r)-b_{x1}^*(r)]^2|\hat{X_1}(r)|^2 dr\bigg)^{\frac{p}{2}}\bigg]\\
  &\quad+C\mathbb{E}^{\mathcal{F}^s_t}\bigg[\bigg(\int_t^T[b_{x2}^\theta(r)-b_{x2}^*(r)]^2|\hat{X_2}(r)|^2 dr\bigg)^{\frac{p}{2}}\bigg]\\
 =&\ o(|x'-X^*(t)|^p),\ \mathbb{P}\mbox{-}a.s.
\end{aligned}\end{eqnarray}
The second equality in (\ref{eq3.10}) can be proved similarly.

$Step\ 3.$ Duality relation.

Denoting $\tilde{p}(t):=p_1(t)(\gamma(t))^{-1}$, $\check{p}(t):=p_2(t)(\gamma(t))^{-1}$, then applying It\^o's formula, we can derive
\begin{equation}\left\{\begin{aligned}\label{eq3.11}
  d\tilde{p}(t)&=\big\{f_x^*(t)-\tilde{p}(t)[b_x^*(t)+f_y^*(t)+\sigma_x^*(t)f_z^*(t)]-\tilde{q}(t)[\sigma_x^*(t)+f_z^*(t)]-\check{p}(t)\big\}dt\\
               &\quad+\tilde{q}(t)dW(t),\ t\in[s,T],\\
  \tilde{p}(T)&=-\phi_x(X^*(T),X_1^*(T)),
\end{aligned}\right.\end{equation}
\begin{equation}\left\{\begin{aligned}\label{eq3.12}
  d\check{p}(t)&=\big\{\check{p}(t)[\lambda-f_y^*(t)]-\check{q}(t)f_z^*(t)+f_{x1}^*(t)-\tilde{p}(t)[b_{x1}^*(t)+f_z^*(t)\sigma_{x1}^*(t)]\\
               &\qquad-\tilde{q}(t)\sigma_{x1}^*(t)\big\}dt+\check{q}(t)dW(t),\ t\in[s,T],\\
   \check{p}(T)&=-\phi_{x1}(X^*(T),X_1^*(T)),
\end{aligned}\right.\end{equation}
where $\tilde{q}(t):=q_1(t)(\gamma(t))^{-1}-\tilde{p}(t)f_z^*(t)$, $\check{q}(t):=q_2(t)(\gamma(t))^{-1}-\check{p}(t)f_z^*(t)$.

Obviously, $(\tilde{p}(\cdot),\tilde{q}(\cdot))\in\mathcal{S}_\mathcal{F}^2([s,T];\mathbf{R})\times L_\mathcal{F}^2([s,T];\mathbf{R})$, and $(\check{p}(\cdot),\check{q}(\cdot))\in \mathcal{S}_\mathcal{F}^2([s,T];\mathbf{R})\times L_\mathcal{F}^2([s,T];\mathbf{R})$.

Denoting $\hat{Y}(r):=\tilde{p}(r)\hat{X}(r)$, $\check{Y}(r):=\check{p}(r)\hat{X_1}(r)$, and applying It\^o's formula to $\hat{Y}(\cdot)$, by (\ref{eq3.4}), (\ref{eq3.11}), we obtain
\begin{eqnarray}\label{eq3.13}\left\{\begin{aligned}
d\hat{Y}(r)&=A(r)dr+\hat{Z}(r)dW(r),\ r\in[t,T],\\
 \hat{Y}(T)&=-\phi_x(X^*(T),X_1^*(T))\hat{X}(T),
\end{aligned}\right.\end{eqnarray}
where
\begin{eqnarray}\nonumber\begin{aligned}
      A(r)&:=\hat{X}(r)\big[f_x^*(r)-\tilde{p}(r)f_y^*(r)-\tilde{p}(r)\sigma_x^*(r)f_z^*(r)-\tilde{q}(r)f_z^*(r)-\check{p}(t)\big]\\
          &\quad+\tilde{p}(r)\big[b_{x_1}^*(r)\hat{X_1}(r)+b_{x_2}^*(r)\hat{X_2}(r)+\varepsilon_1(r)\big]\\
          &\quad+\tilde{q}(r)[\sigma_{x_1}^*(r)\hat{X_1}(r)+\sigma^*_{x_2}(r)\hat{X_2}(r)+\varepsilon_2(r)],\\
\hat{Z}(r)&:=\hat{X}(r)\tilde{q}(r)+\hat{X}(r)\tilde{p}(r)\sigma_x^*(r)+\tilde{p}(r)\hat{X_1}(r)\sigma_{x1}^*(r)+\tilde{p}(r)\sigma_{x2}^*(r)\hat{X_2}(r)+\tilde{p}(r)\varepsilon_2(r).
\end{aligned}\end{eqnarray}
Noting $d\hat{X_1}(r)=[\hat{X}(r)-\lambda\hat{X_1}(r)-e^{-\lambda\delta}\hat{X_2}(r)]dr$, then applying It\^o's formula to $\check{Y}(\cdot)$, we get
\begin{eqnarray}\label{eq3.14}\left\{\begin{aligned}
d\check{Y}(r)&=B(r)dr+\check{Z}(r)dW(r),\ r\in[t,T],\\
 \check{Y}(T)&=-\phi_{x1}(X^*(T),X_1^*(T))\hat{X_1}(T),
\end{aligned}\right.\end{eqnarray}
where\begin{eqnarray}\nonumber\begin{aligned}
        B(r)&:=\check{p}(r)\hat{X}(r)-e^{-\lambda\delta}\check{p}(r)\hat{X_2}(r)+\hat{X_1}(r)\big[-\check{p}(r)f_y^*(r)-\check{q}(r)f_z^*(r)+f_{x1}^*(r)\\
            &\quad-\tilde{p}(r)b_{x1}^*(r)-\tilde{p}(r)f_z^*(r)\sigma_{x1}^*(r)-\tilde{q}(r)\sigma_{x1}^*(r)\big]\\
\check{Z}(r)&:=\check{q}(r)\hat{X_1}(r).
\end{aligned}\end{eqnarray}

$Step\ 4.$ Variational equation for BSDE.

For the above $x'\in\mathbf{R}$, recall that $X^{t,x',X_1^*(t)}(\cdot)$ is given by (\ref{eq3.2}) and denote by $(Y^{t,x',X_1^*(t)}(\cdot),\\Z^{t,x',X_1^*(t)}(\cdot))$ the solution to the following BSDE on $[t,T]$:
\begin{eqnarray}\label{eq3.15}\left\{\begin{aligned}
 -dY^{t,x',X_1^*(t)}(r)&=f\big(r,X^{t,x',X_1^*(t)}(r),X_1^{t,x',X_1^*(t)}(r),X_2^{t,x',X_1^*(t)}(r),\\
                       &\qquad Y^{t,x',X_1^*(t)}(r),Z^{t,x',X_1^*(t)}(r),u^*(r)\big)dr-Z^{t,x',X_1^*(t)}(r)dW(r),\\
   Y^{t,x',X_1^*(t)}(T)&=\phi(X^{t,x',X_1^*(t)}(T),X_1^{t,x',x_1^*(t)}(T)).
\end{aligned}\right.\end{eqnarray}
Similarly BSDE (\ref{eq3.15}) is defined on $(\Omega,\mathcal{F},\{\mathcal{F}^s_r\}_{r\geq s},\mathbb{P}(\cdot|\mathcal{F}^s_t)(\omega))$ for $\mathbb{P}$-$a.s.\ \omega.$

For any $t\leq r \leq T$, set
\begin{eqnarray}\begin{aligned}\label{eq3.16}
   &\tilde{Y}(r):=-Y^{t,x',X_1^*(t)}(r)+Y^*(r)-\hat{Y}(r)-\check{Y}(r), \\
   &\tilde{Z}(r):=-Z^{t,x',X_1^*(t)}(r)+Z^*(r)-\hat{Z}(r)-\check{Z}(r).
\end{aligned}\end{eqnarray}
Thus by (\ref{eq3.13}), (\ref{eq3.14}) and (\ref{eq3.15}), we get
\begin{eqnarray}\label{eq3.17}\left\{\begin{aligned}
  d\tilde{Y}(r)=&\big[f\big(r,X^{t,x',X_1^*(t)}(r),X_1^{t,x',X_1^*(t)}(r),X_2^{t,x',X_1^*(t)}(r),Y^{t,x',X_1^*(t)}(r),\\
                &\quad Z^{t,x',X_1^*(t)}(r),u^*(r)\big)-f^*(r)-A(r)-B(r)\big]dr+\tilde{Z}(r)dW(r), \\
   \tilde{Y}(T)=&-(\hat{X}(T),\hat{X_1}(T))\tilde{D}^2\phi(T)(\hat{X}(T),\hat{X_1}(T))^\top,
\end{aligned}\right.\end{eqnarray}
where $\tilde{D}^2\phi(T):=\int_0^1\int_0^1\lambda D^2\phi(X^*(T)+\lambda\theta\hat{X}(T),X_1^*(T)+\lambda\theta\hat{X_1}(T))d\lambda d\theta$.

By the boundedness of $\tilde{D}^2\phi$, we have
\begin{equation}\label{eq3.18}\begin{aligned}
  \mathbb{E}^{\mathcal{F}^s_t}|\tilde{Y}(T)|^2\leq C\mathbb{E}^{\mathcal{F}_t^s}\big[|\hat{X}(T)|^4+|\hat{X_1}(T)|^4\big]=o(|x'-X^*(t)|^2),\ \mathbb{P}\mbox{-}a.s.
\end{aligned}\end{equation}
Noting (\ref{eq3.16}), we have
\begin{equation*}\begin{aligned}\begin{aligned}
  &\quad f\big(r,X^{t,x',X_1^*(t)}(r),X_1^{t,x',X_1^*(t)}(r),X_2^{t,x',X_1^*(t)}(r),Y^{t,x',X_1^*(t)}(r),Z^{t,x',X_1^*(t)}(r),u^*(r)\big)-f^*(r)\\
  &=f(r,X^{t,x',X_1^*(t)}(r),X_1^{t,x',X_1^*(t)}(r),X_2^{t,x',X_1^*(t)}(r),Y^{t,x',X_1^*(t)}(r),Z^{t,x',X_1^*(t)}(r),u^*(r))\\
  &\quad-f(r,X^*(r)+\hat{X}(r),X_1^*(r)+\hat{X_1}(r),X_2^*(r)+\hat{X_2}(r),Y^*(r)-\hat{Y}(r)-\check{Y}(r),\\
  &\qquad\quad Z^*(r)-\hat{Z}(r)-\check{Z}(r),u^*(r))\\
  &\quad+f(r,X^*(r)+\hat{X}(r),X_1^*(r)+\hat{X_1}(r),X_2^*(r)+\hat{X_2}(r),Y^*(r)-\hat{Y}(r)-\check{Y}(r),\\
  &\qquad\quad Z^*(r)-\hat{Z}(r)-\check{Z}(r),u^*(r))-f^*(r)\\
  &=-\tilde{f}_y(r)\tilde{Y}(r)-\tilde{f}_z(r)\tilde{Z}(r)+f_x^*(r)\hat{X}(r)+f_{x_1}^*(r)\hat{X_1}(r)+f_{x_2}^*(r)\hat{X_2}(r)\\
  &\quad-f_y^*(r)[\hat{Y}(r)+\check{Y}(r)]-f_z^*(r)[\hat{Z}(r)+\check{Z}(r)]\\
  &\quad+[\hat{X}(r),\hat{X_1}(r),\hat{X_2}(r),-\hat{Y}(r)-\check{Y}(r),-\hat{Z}(r)-\check{Z}(r)]\tilde{D}^2f(r)\\
  &\qquad\times[\hat{X}(r),\hat{X_1}(r),\hat{X_2}(r),-\hat{Y}(r)-\check{Y}(r),-\hat{Z}(r)-\check{Z}(r)]^\top,
\end{aligned}\end{aligned}\end{equation*}
where
\begin{eqnarray*}\left\{\begin{aligned}
  \tilde{f}_y(r)=&\int_0^1f_y(r,X^*(r)+\hat{X}(r),X_1^*(r)+\hat{X_1}(r),X_2^*(r)+\hat{X_2}(r),Y^*(r)\\
                 &\qquad\quad-\hat{Y}(r)-\check{Y}(r)-\theta \tilde{Y}(r),Z^*(r)-\hat{Z}(r)-\check{Z}(r)-\theta\tilde{Z}(r),u^*(r))d\theta,\\
  \tilde{f}_z(r)=&\int_0^1f_z(r,X^*(r)+\hat{X}(r),X_1^*(r)+\hat{X_1}(r),X_2^*(r)+\hat{X_2}(r),Y^*(r)\\
                 &\qquad\quad-\hat{Y}(r)-\check{Y}(r)-\theta \tilde{Y}(r),Z^*(r)-\hat{Z}(r)-\check{Z}(r)-\theta\tilde{Z}(r),u^*(r))d\theta,\\
 \tilde{D}^2f(r)=&\int_0^1\int_{0}^{1}\lambda D^2f(r,X^*(r)+\lambda\theta \hat{X}(r),X_1^*(r)+\lambda\theta\hat{X_1}(r),X_2^*(r)+\lambda\theta \hat{X_2}(r),\\
                 &\qquad\quad Y^*(r)-\lambda\theta\hat{Y}(r)-\lambda\theta\check{Y}(r),Z^*(r)-\lambda\theta\hat{Z}(r)-\lambda\theta\check{Z}(r),u^*(r))d\lambda d\theta.
\end{aligned}\right.\end{eqnarray*}
Hence, we obtain
\begin{equation*}\begin{aligned}
\tilde{Y}(t)&=\tilde{Y}(T)+\int_{t}^{T}\Big\{\tilde{f}_y(r)\tilde{Y}(r)+\tilde{f}_z(r)\tilde{Z}(r)+\tilde{p}(r)\varepsilon_1(r)+[\tilde{p}(r)f_z^*(r)+\tilde{q}(r)]\varepsilon_2(r)\\
            &\qquad+\hat{X_2}(r)[\tilde{p}(r)b_{x_2}^*(r)+\tilde{q}(r)\sigma_{x_2}^*(r)-e^{-\lambda\delta}\check{p}(r)+\tilde{p}(r)\sigma_{x_2}^*(r)f_z^*(r)-f_{x_2}^*(r)]\\
            &\qquad-[\hat{X}(r),\hat{X_1}(r),\hat{X_2}(r),-\hat{Y}(r)-\check{Y}(r),-\hat{Z}(r)-\check{Z}(r)]\tilde{D}^2f(r)\\
            &\qquad\times[\hat{X}(r),\hat{X_1}(r),\hat{X_2}(r),-\hat{Y}(r)-\check{Y}(r),-\hat{Z}(r)-\check{Z}(r)]^\top\Big\}dr-\int_t^T\tilde{Z}(r)dW(r).
\end{aligned}\end{equation*}
Applying It\^o's formula to $p_3(r)(\gamma(r))^{-1}$, we have
\begin{equation*}\begin{aligned}
  d[p_3(r)(\gamma(r))^{-1}]=&\big\{-\tilde{p}(r)b_{x_2}^*(r)-\tilde{q}(r)\sigma_{x_2}^*(r)+e^{-\lambda\delta}\check{p}(r)-\tilde{p}(r)\sigma_{x_2}^*(r)f_z^*(r)\\
                            &\ +f_{x_2}^*(r)-p_3(r)(\gamma(r))^{-1}f_y^*(t)+p_3(r)(\gamma(r))^{-1}|f_z^*(r)|^2\big\}dr\\
                            &-p_3(r)(\gamma(r))^{-1}f_z^*(r)dW(r).
\end{aligned}\end{equation*}
Substituting $p_3(r)\equiv 0$ into the above equality, we get
\begin{equation*}\begin{aligned}
  0=&\int_{t+\delta}^T\big\{-\tilde{p}(r)b_{x_2}^*(r)-\tilde{q}(r)\sigma_{x_2}^*(r)+e^{-\lambda\delta}\check{p}(r)-\tilde{p}(r)\sigma_{x_2}^*(r)f_z^*(r)+f_{x_2}^*(r)\big\}dr\\
   =&\int_t^{T-\delta}\big\{-\tilde{p}(r+\delta)b_{x_2}^*(r+\delta)-\tilde{q}(r+\delta)\sigma_{x_2}^*(r+\delta)+e^{-\lambda\delta}\check{p}(r+\delta)\\
    &\qquad\quad-\tilde{p}(r+\delta)\sigma_{x_2}^*(r+\delta)f_z^*(r+\delta)+f_{x_2}^*(r+\delta)\big\}dr.
\end{aligned}\end{equation*}
Recalling (\ref{eq3.4}), we have
\begin{equation*}\begin{aligned}
  \hat{X}_2(T)&=\hat{X}(t)+\int_t^{T-\delta}\big\{b^*_x(r)\hat{X}(r)+b_{x_1}^*(r)\hat{X_1}(r)+b_{x_2}^*(r)\hat{X_2}(r)+\varepsilon_1(r)\big\}dr\\
              &\quad+\int_t^{T-\delta}\big\{\sigma_x^*(r)\hat{X}(r)+\sigma_{x_1}^*(r)\hat{X_1}(r)+\sigma^*_{x_2}(r)\hat{X_2}(r)+\varepsilon_2(r)\big\}dW(r).
\end{aligned}\end{equation*}
Apparently $\hat{X}_2(T)$ is $\mathcal{F}_{T-\delta}$-adapted, under the filtration $\mathcal{F}_{T-\delta}$, applying It\^o's formula to $0\cdot\hat{X}_2(\cdot)$, by the above two equalities we derive
\begin{equation*}
  0=\int_{t+\delta}^T\hat{X_2}(r)\big[\tilde{p}(r)b_{x_2}^*(r)+\tilde{q}(r)\sigma_{x_2}^*(r)-e^{-\lambda\delta}\check{p}(r)+\tilde{p}(r)\sigma_{x_2}^*(r)f_z^*(r)-f_{x_2}^*(r)\big]dr.
\end{equation*}
Noting $\hat{X}_2(r)=0$ for $r\in[t,t+\delta)$, we have
\begin{equation*}
  0=\int_t^T\hat{X_2}(r)\big[\tilde{p}(r)b_{x_2}^*(r)+\tilde{q}(r)\sigma_{x_2}^*(r)-e^{-\lambda\delta}\check{p}(r)+\tilde{p}(r)\sigma_{x_2}^*(r)f_z^*(r)-f_{x_2}^*(r)\big]dr.
\end{equation*}
Thus, we deduce
\begin{equation}\label{eq3.19}\begin{aligned}
\tilde{Y}(t)&=\tilde{Y}(T)+\int_t^T\Big\{\tilde{f}_y(r)\tilde{Y}(r)+\tilde{f}_z(r)\tilde{Z}(r)+\tilde{p}(r)\varepsilon_1(r)+[\tilde{p}(r)f_z^*(r)+\tilde{q}(r)]\varepsilon_2(r)\\
            &\qquad-[\hat{X}(r),\hat{X_1}(r),\hat{X_2}(r),-\hat{Y}(r)-\check{Y}(r),-\hat{Z}(r)-\check{Z}(r)]\tilde{D}^2f(r)\\
            &\qquad\times[\hat{X}(r),\hat{X_1}(r),\hat{X_2}(r),-\hat{Y}(r)-\check{Y}(r),-\hat{Z}(r)-\check{Z}(r)]^\top\Big\}dr-\int_{t}^{T}\tilde{Z}(r)dW(r).
\end{aligned}\end{equation}

$Step\ 5$. Estimate of remainder terms of BSDE.

Next we use Lemma \ref{lem3.1} to derive the the estimate of solution to the BSDE (\ref{eq3.17}). By (\ref{eq3.19}), we have
\begin{equation}\begin{aligned}\label{eq3.20}
  &\mathbb{E}^{\mathcal{F}_t^s}\Big[\sup\limits_{t\leq r\leq T}|\tilde{Y}(r)|^2\Big]\leq o(|x'-X^*(t)|^2)
   +C\mathbb{E}^{\mathcal{F}_t^s}\bigg[\bigg(\int_t^T\tilde{p}(r)\varepsilon_1(r)dr\bigg)^2\\
  &\quad+\bigg(\int_t^T[\tilde{p}(r)f_z^*(r)+\tilde{q}(r)]\varepsilon_2(r)dr\bigg)^2+\bigg(\int_t^T[\hat{X}(r),\hat{X_1}(r),\hat{X_2}(r),-\hat{Y}(r)-\check{Y}(r),\\
  &\qquad-\hat{Z}(r)-\check{Z}(r)]\tilde{D}^2f(r)[\hat{X}(r),\hat{X_1}(r),\hat{X_2}(r),-\hat{Y}(r)-\check{Y}(r),-\hat{Z}(r)-\check{Z}(r)]^Tdr\bigg)^2\bigg]\\
  &:=o(|x'-X^*(t)|^2)+I+II+III,\ \mathbb{P}\mbox{-}a.s.
\end{aligned}\end{equation}

First, applying (\ref{eq3.10}) we have
\begin{equation}\begin{aligned}\label{eq3.21}
  I&\leq C\bigg\{\mathbb{E}^{\mathcal{F}_t^s}\Big[\int_t^T|\tilde{p}(r)|^2dr\Big]^2\bigg\}^{\frac{1}{2}}
    \bigg\{\mathbb{E}^{\mathcal{F}_t^s}\Big[\int_t^T|\varepsilon_1(r)|^2dr\Big]^2\bigg\}^{\frac{1}{2}}=o(|x'-X^*(t)|^2),\ \mathbb{P}\mbox{-}a.s.
\end{aligned}\end{equation}
In the same method, by the boundedness of $f_z$, we obtain
\begin{equation}\begin{aligned}\label{eq3.22}
II&=\mathbb{E}^{\mathcal{F}_t^s}\bigg[\bigg(\int_t^T[\tilde{p}(r)f_z^*(r)+\tilde{q}(r)]\varepsilon_2(r)dr\bigg)^2\bigg]\\
  &\leq C\bigg\{\mathbb{E}^{\mathcal{F}_t^s}\Big[\int_t^T|\tilde{p}(r)|^2dr\Big]^2\bigg\}^{\frac{1}{2}}
   \bigg\{\mathbb{E}^{\mathcal{F}_t^s}\Big[\int_t^T|\varepsilon_2(r)|^2dr)\Big]^2\bigg\}^{\frac{1}{2}}\\
  &\quad+C\bigg\{\mathbb{E}^{\mathcal{F}_t^s}\Big[\int_t^T|\tilde{q}(r)|^2dr\Big]^2\bigg\}^{\frac{1}{2}}
   \bigg\{\mathbb{E}^{\mathcal{F}_t^s}\Big[\int_t^T|\varepsilon_2(r)|^2dr\Big]^2\bigg\}^{\frac{1}{2}}=o(|x'-X^*(t)|^2),\ \mathbb{P}\mbox{-}a.s.
\end{aligned}\end{equation}
Finally, recall (\ref{eq3.8}), we get
\begin{equation*}\begin{aligned}
  &\mathbb{E}^{\mathcal{F}_t^s}\bigg[\bigg(\int_t^T|\check{Z}(r)|^2dr\bigg)^{\frac{p}{2}}\bigg]
   =\mathbb{E}^{\mathcal{F}_t^s}\bigg[\bigg(\int_t^T|\check{q}(r)\hat{X_1}(r)|^2dr\bigg)^{\frac{p}{2}}\bigg]\\
  &\leq\bigg\{\mathbb{E}^{\mathcal{F}_t^s}\Big[\int_t^T|\check{q}(r)|^2dr\Big]^p\bigg\}^{\frac{1}{2}}
   \bigg\{\mathbb{E}^{\mathcal{F}_t^s}\Big[\sup\limits_{t\leq r\leq T}|\hat{X_1}(r)|^{2p}\Big]\bigg\}^{\frac{1}{2}}\leq C(|x'-X^*(t)|^p),\ \mathbb{P}\mbox{-}a.s.
\end{aligned}\end{equation*}
By the boundedness of $\sigma_x, \sigma_{x_1}, \sigma_{x_2}$ and (\ref{eq3.7}), (\ref{eq3.10}), we have
\begin{equation*}\begin{aligned}
  \mathbb{E}^{\mathcal{F}_t^s}\bigg[\bigg(\int_t^T|\hat{Z}(r)|^2dr\bigg)^{\frac{p}{2}}\bigg]\leq C(|x'-X^*(t)|^p),\ \mathbb{P}\mbox{-}a.s.
\end{aligned}\end{equation*}
Using the above two inequlities and the boundedness of $\tilde{D}^2f$, we derive
\begin{equation}\begin{aligned}\label{eq3.23}
  III&\leq C\mathbb{E}^{\mathcal{F}_t^s}\Big[\int_t^T\Big(|\hat{X}(r)|^2+|\hat{X_1}(r)|^2+|\hat{X_2}(r)|^2+|\hat{Y}(r)|^2+|\check{Y}(r)|^2\\
     &\qquad\qquad\qquad+|\hat{Z}(r)|^2+|\check{Z}(r)|^2\Big)dr\Big]^2=o(|x'-X^*(t)|^2),\ \mathbb{P}\mbox{-}a.s.
\end{aligned}\end{equation}
Hence, substituting (\ref{eq3.21}), (\ref{eq3.22}) and (\ref{eq3.23}) into (\ref{eq3.20}), the following estimate hold:
\begin{equation}\begin{aligned}\label{eq3.24}
  \mathbb{E}^{\mathcal{F}_t^s}\Big[\sup\limits_{t\leq r\leq T}|\tilde{Y}(r)|\Big]=o(|x'-X^*(t)|),\ \mathbb{P}\mbox{-}a.s.
\end{aligned}\end{equation}

$Step\ 6$. Completion of the proof.

Suppose $x'\in \mathbf{R}$ is rational. Since they are countable, we can find a subset $\Omega_0\subseteq \Omega$ with $\mathbb{P}(\Omega_0)=1$ such that for any $\omega_0\in\Omega_0$,
\begin{equation*}\left\{\begin{aligned}
  &V(t,X^*(t,\omega_0),X_1^*(t,\omega_0))=-Y^*(t,\omega_0),\ (\ref{eq3.7}),\ (\ref{eq3.8}),\ (\ref{eq3.9}),\ (\ref{eq3.10}),\ (\ref{eq3.13}),\\
  &\ (\ref{eq3.14}),\ (\ref{eq3.18}),\ (\ref{eq3.20}),\ (\ref{eq3.24})\mbox{ are satisfied for any rational }x',\\
  &\ (\Omega,\mathcal{F},\mathbb{P}(\cdot|\mathcal{F}_t^{s})(\omega_0),W(\cdot)-W(t);u(\cdot)|_{[t,T]})\in\mathcal{U}^\omega[t,T].
\end{aligned}\right.\end{equation*}
The first relation can be obtained by the DPP (see \cite{CW12}). Let $\omega_0\in\Omega_0$ be fixed, and for any rational $x'\in\mathbf{R}$, by (\ref{eq3.24}), we have
\begin{equation}\label{eq3.25}
  \tilde{Y}(t,\omega_0)=o(|x'-X^*(t,\omega_0)|),\mbox{ for all } t\in[s,T].
\end{equation}
By the definition of $\tilde{Y}(\cdot)$, we have
\begin{equation}\nonumber\begin{aligned}
  -Y^{t,x',X_1^*(t)}(t,\omega_0)+Y^*(t,\omega_0)\leq\tilde{p}(t,\omega_0)(x'-X^*(t,\omega_0))+o(|x'-X^*(t,\omega_0)|),
\end{aligned}\end{equation}
for all $t\in[s,T]$. Thus
\begin{equation}\label{eq3.26}\begin{aligned}
  &V(t,x',X_1^*(t,\omega_0))-V(t,X^*(t,\omega_0),X_1^*(t,\omega_0))\leq-Y^{t,x',X_1^*(t)}(t,\omega_0)+Y^*(t,\omega_0)\\
  &\leq p_1(t,\omega_0)(\gamma(t,\omega_0))^{-1}(x'-X^*(t,\omega_0))+o(|x'-X^*(t,\omega_0)|),\mbox{ for all } t\in[s,T].
\end{aligned}\end{equation}
The above term $o(|x'-X^*(t,\omega_0)|)$ depends only on the size of $|x'-X^*(t,\omega_0)|$, and is independent of $x'$. Therefore, by the continuity of $V(t,\cdot,\cdot)$, (\ref{eq3.26}) holds for all $x'\in\mathbf{R}$, which proves
\begin{eqnarray}\nonumber
  \big\{p_1(t)(\gamma(t))^{-1}\big\}\subseteq D^{1,+}_{x}V(t,X^*(t),X_1^*(t)),\mbox{ for all } t\in[s,T],\ \mathbb{P}\mbox{-}a.s.
\end{eqnarray}
Finally, fix an $\omega\in\Omega$ such that (\ref{eq3.26}) holds for any $x'\in\mathbf{R}$. For any $p\in D^{1,-}_{x}V(t,X^*(t),X_1^*(t))$, by definition of $D^{1,-}_{x}V(t,X^*(t),X_1^*(t))$ we have
\begin{equation}\nonumber
  \varliminf_{x'\rightarrow X^*(t)}\frac{V(t,x',X_1^*(t))-V(t,X^*(t),X_1^*(t))-p(x'-X^*(t))}{|x'-X^*(t)|}\geq 0.
\end{equation}
Since $p_1(t)(\gamma(t))^{-1}\in D^{1,+}_{x}V(t,X^*(t),X_1^*(t))$, we obtain
\begin{equation}\nonumber
  \varliminf_{x'\rightarrow X^*(t)}\frac{(p-p_1(t)(\gamma(t))^{-1})(x'-X^*(t))}{|x'-X^*(t)|}\leq 0.
\end{equation}
Thus, it is necessary that
\begin{equation}\nonumber
  p=p_1(t)(\gamma(t))^{-1},\mbox{ for all } t\in[s,T],\ \mathbb{P}\mbox{-}a.s.
\end{equation}
Thus, the second inclusion of (\ref{eq3.1}) holds. The proof is complete.
\end{proof}

\noindent\textbf{Remark 4.1.} If the value function is smooth, then by Theorem \ref{th3.1} we have
\begin{equation*}
  V_x(t,X^*(t),X_1^*(t))=p_1(t)(\gamma(t))^{-1}.
\end{equation*}
This conclusion is consistent with the result of Theorem 3.2 in \cite{SXZ15}.

\noindent\textbf{Remark 4.2.} It is worth mentioning that in the same manner we can not deduce that
\begin{equation*}
   \big\{p_2(t)(\gamma(t))^{-1}\big\}\subseteq D_{x_1}^{1,+}V(t,X^*(t),X_1^*(t)),
\end{equation*}
because the condition $\hat{X}(r)=0$ for $r\in[t-\delta,t)$ plays a crucial role in the proof. Without this condition, we can not deduce (\ref{eq3.19}).

\section{Stochastic Verification Theorem}

In this section, we give the verification theorem of \textbf{Problem (P)} in the framework of viscosity solution, which can help us verify whether an admissible pair is the optimal pair. Before stating the main theorem, we first recall a basic result given in \cite{CW12}.

We consider the special version of FBSDDE (\ref{eq2.1}):
\begin{eqnarray}\label{eq5.1}\left\{\begin{aligned}
  dX^{s,\varphi;u}(t)&=b(t,X^{s,\varphi;u}(t),X_1^{s,\varphi;u}(t),X_2^{s,\varphi;u}(t),u(t))dt\\
                     &\quad+\sigma(t,X^{s,\varphi;u}(t),X_1^{s,\varphi;u}(t),X_2^{s,\varphi;u}(t),u(t))dW(t),t\in[s,T],\\
 -dY^{s,\varphi;u}(t)&=\big[a(t,X^{s,\varphi;u}(t),X_1^{s,\varphi;u}(t),X_2^{s,\varphi;u}(t),u(t))\\
                     &\qquad+f(t)Y^{s,\varphi;u}(t)+g(t)Z^{s,\varphi;u}(t)\big]dt-Z^{s,\varphi;u}(t)dW(t),\ t\in[s,T],\\
   X^{s,\varphi;u}(t)&=\varphi(t-s),\ t\in[s-\delta,s],\\
   Y^{s,\varphi;u}(T)&=\phi(X^{s,\varphi;u}(T),X_1^{s,\varphi;u}(T)),
\end{aligned}\right.\end{eqnarray}
where $b,\sigma,\varphi$ satisfy \textbf{(H1)}, $a:[0,T]\times\mathbf{R}\times\mathbf{R}\times\mathbf{R}\times\textbf{U}\rightarrow\mathbf{R}$ and $\phi$ satisfies \textbf{(H2)}. Moreover, $f(\cdot), g(\cdot)$ are given uniformly bounded deterministic function. Then we have the following result.
\begin{mypro}\label{prop 5.1}
Consider \textbf{Problem (P)} with the state equation (\ref{eq5.1}), then the value function $V$ defined by (\ref{eq2.3}) is the unique viscosity solution to the following HJB equation in the class of continuous function with at most a polynomially growth:
\begin{eqnarray}\label{eq5.2}\left\{\begin{aligned}
 -V_s(s,x,x_1)&+\underset{u\in\mathbf{U}}{\sup}\ \bar{G}\big(s,x,x_1,x_2,u,-V(s,x,x_1),-V_x(s,x,x_1),\\
              &\qquad\qquad-V_{xx}(s,x,x_1),-V_{x_1}(s,x,x_1)\big)=0,\ (s,x,x_1)\in[0,T]\times\mathbf{R}^2,\\
    V(T,x,x_1)&=-\phi(x,x_1),\ \mbox{for all }x,x_1\in\mathbf{R},
\end{aligned}\right.\end{eqnarray}
where the generalized Hamiltonian function $\bar{G}:[0,T]\times\mathbf{R}\times\mathbf{R}\times\mathbf{R}\times\mathbf{U}\times\mathbf{R}\times\mathbf{R}\times\mathbf{R}\times\mathbf{R}\rightarrow \mathbf{R}$ is defined as
\begin{equation}\label{eq5.3}\begin{split}
&\bar{G}(s,x,x_1,x_2,u,k,p,R,q):=\frac{1}{2}R\sigma^2(s,x,x_1,x_2,u)+p\big[b(s,x,x_1,x_2,u)\\
&\quad+\sigma(s,x,x_1,x_2,u)g(s)\big]+q(x-\lambda x_1-e^{-\lambda\delta }x_2)+a(s,x,x_1,x_2,u)+f(s)k.
\end{split}\end{equation}
\end{mypro}
Next we give a lemma which can be used in the proof of Theorem \ref{thm5.1}.
\begin{mylem}\label{lem5.1} \emph{(see \cite{YZ99})}
Let $\lambda\in C[0,T]$. Extend $\lambda$ to $(-\infty,+\infty)$ with $\lambda(t)=\lambda(T)$ for $t>T$, and $\lambda(t)=\lambda(0)$ for $t<0$. Suppose there is a $\rho\in L^1(0,T;\mathbf{R})$ and some $h_0>0$ such that
\begin{equation*}
  \frac{\lambda(t+h)-\lambda(t)}{h}\leq\rho(t),\ a.e.\ t\in [0,T],\ h\leq h_0.
\end{equation*}
Then
\begin{equation*}
  \lambda(\beta)-\lambda(\alpha)\leq\int_\alpha^\beta\underset{h\rightarrow 0+}\varlimsup\frac{\lambda(t+h)-\lambda(t)}{h}dt,\ \forall\ 0\leq\alpha\leq\beta\leq T.
\end{equation*}
\end{mylem}
The main theorem in this section is as follows.
\begin{mythm}\label{thm5.1}
Suppose $b,\sigma,\varphi$ satisfy \textbf{(H1)}, $a,\phi$ satisfy \textbf{(H2)}, $g(\cdot)\equiv 0$ and $f(\cdot)$ is a given uniformly bounded deterministic function. Let $V\in C([0,T]\times\mathbf{R}^2)$ be the viscosity solution to the following HJB equation satisfying $|V(t,x,x_1)|\leq C(1+|x|^k+|x_1|^k)$, for some $k\geq1,(t,x,x_1)\in(0,T)\times\mathbf{R}^2$:
\begin{eqnarray*}\left\{\begin{aligned}
 -V_s(s,x,x_1)&+\underset{u\in\mathbf{U}}{\sup}\ G(s,x,x_1,x_2,u,-V(s,x,x_1),-V_x(s,x,x_1),\\
              &-V_{xx}(s,x,x_1),-V_{x_1}(s,x,x_1))=0,\ (s,x,x_1)\in[0,T]\times\mathbf{R}^2,\\
    V(T,x,x_1)&=-\phi(x,x_1),\ \mbox{for all }x,x_1\in\mathbf{R},
\end{aligned}\right.\end{eqnarray*}
where the generalized Hamiltonian function $\tilde{G}:[0,T]\times\mathbf{R}\times\mathbf{R}\times\mathbf{R}\times\mathbf{U}\times\mathbf{R}\times\mathbf{R}\times\mathbf{R}\times\mathbf{R}\rightarrow \mathbf{R}$ is defined as
\begin{equation}\begin{split}
\tilde{G}(s,x,x_1,x_2,u,k,p,R,q)&:=\frac{1}{2}R\sigma^2(s,x,x_1,x_2,u)+pb(s,x,x_1,x_2,u)\\
                                &\quad+q(x-\lambda x_1-e^{-\lambda\delta }x_2)+a(s,x,x_1,x_2,u)+f(s)k.
\end{split}\end{equation}
Then we have
\begin{equation}\label{eq5.4}
  V(s,x,x_1)\leq J(s,\varphi;u(\cdot)),\mbox{ for all }(s,\varphi)\in[0,T]\times C([-\delta,0];\mathbf{R})\mbox{ and any }u(\cdot)\in\mathcal{U}^{w}(s,T).
\end{equation}
Furthermore, let $(s,x,x_1,x_2)\in[0,T]\times\mathbf{R}^3$ be fixed, suppose $(X^*(\cdot),Y^*(\cdot),Z^*(\cdot),u^*(\cdot))$ is an admissible pair such that there exists a quadruple $(\Theta,p,q,P)\in L_\mathcal{F}^2(s,T;\mathbf{R})\times L_\mathcal{F}^2(s,T;\mathbf{R})\times L_\mathcal{F}^2(s,T;\mathbf{R})\times L_\mathcal{F}^2(s,T;\mathbf{R})$ satisfying
 \begin{equation}\label{eq5.5}
   (\Theta,p,q,P)\in D_{t+,x}^{1,2,1,+}V(t,X^*(t),X_1^*(t)),\mbox{ for }a.e.\ t\in[s,T],\ \mathbb{P}\mbox{-}a.s.,
   \end{equation}
and
\begin{equation}\label{eq5.6}\begin{aligned}
  &\mathbb{E}\int_s^T\big[\Theta(t)-\tilde{G}\big(t,X^*(t),X_1^*(t),X^*_2(t),u^*(t),-V(t,X^*(t),X_1^*(t)),\\
  &\qquad\quad-p(t),-P(t),-q(t)\big)\big]dt\leq 0.
\end{aligned}\end{equation}
Then $(X^*(\cdot),Y^*(\cdot),Z^*(\cdot),u^*(\cdot))$ is the optimal pair.
\end{mythm}
\begin{proof}
The first part can be proved by the uniqueness of the viscosity solution by proposition \ref{prop 5.1}. Next we try to prove the second part.

Suppose $(t_0,\omega_0)\in[s,T]\times\Omega$ satisfy (\ref{eq5.5}). Without loss of generality, suppose $\mathbb{P}(\cdot|\mathcal{F}^s_{t_0})(\omega_0)$ is a probability measure. Consider the new probability space $(\Omega,\mathcal{F},\mathbb{P}(\cdot|\mathcal{F}^s_{t_0})(\omega_0))$, under this probability space $\tilde{W}(\cdot)=W(\cdot)-W(t_0)$ still is a standard Brownian motion defined on $[t_0,T]$. We denote $\{\mathcal{F}_t^{t_0}\}_{t\geq t_0}$ the natural filtration generated by $\tilde{W}(\cdot)$, then $u(\cdot)|_{[t_0,T]}$ is $\{\mathcal{F}_t^{t_0}\}$-adapted. Under the new probability space $(\Omega,\mathcal{F},\mathbb{P}(\cdot|\mathcal{F}^s_{t_0})(\omega_0))$ we reconsider the FBSMDDE (\ref{eq5.1}) and it becomes the following form:
\begin{eqnarray}\label{eq5.7}
\left\{\begin{aligned}
  dX^{t_0,x_0;u}(t)&=b(t,X^{t_0,x_0;u}(t),X_1^{t_0,x_0;u}(t),X_2^{t_0,x_0;u}(t),u(t))dt\\
                   &\quad+\sigma(t,X^{t_0,x_0;u}(t),X_1^{t_0,x_0;u}(t),X_2^{t_0,x_0;u}(t),u(t))d\tilde{W}(t),\ t\in[t_0,T],\\
 -dY^{t_0,x_0;u}(t)&=\big[a(t,X^{t_0,x_0;u}(t),X_1^{t_0,x_0;u}(t),X_2^{t_0,x_0;u}(t),u(t))+f(t)Y^{t_0,x_0;u}(t)\big]dt\\
                   &\quad-Z^{t_0,x_0;u}(t)d\tilde{W}(t),\ t\in[t_0,T],\\
   X^{t_0,x_0;u}(t)&=x_0(t),\ t\in[t_0-\delta,t_0],\\
   Y^{t_0,x_0;u}(T)&=\phi(X^{t_0,x_0;u}(T),X_1^{t_0,x_0;u}(T)),
\end{aligned}\right.
\end{eqnarray}
where $x_0(t):= X^{s,\varphi;u}(t)$, $X_1^{t_0,x_0;u}(t):=\int_{-\delta}^0e^{\lambda\tau}X^{t_0,x_0;u}(t+\tau)d\tau$, $X_2^{t_0,x_0;u}(t):=X^{t_0,x_0;u}(t-\delta)$. Denote $\mathbb{E}^{\omega_0}$ the corresponding expectation under the probability measure $\mathbb{P}(\cdot|\mathcal{F}^s_{t_0})(\omega_0)$. The original function space $L_{\mathcal{F}}^p,S_{\mathcal{F}}^p$ is changed to $L_{\mathcal{F}_{{\omega}_0}}^p,S_{\mathcal{F}_{{\omega}_0}}^p$ correspondingly. By Lemma \ref{lem2.2} and Lemma \ref{lem3.1}, we have the following estimate for FBSMDDE (\ref{eq5.7}):
\begin{equation}\label{eq5.8}\begin{aligned}
  &\mathbb{E}^{\omega_0}\bigg[\underset{r\in[t_0,T]}\sup\big[|X^{t_0,x_0;u}(r)|^p+|Y^{t_0,x_0;u}(r)|^p\big]+\Big(\int_{t_0}^T|Z^{t_0,x_0;u}(r)|^2dr\Big)^{\frac{p}{2}}\bigg]\\
  &\leq C[1+\underset{r\in[t_0-\delta,t_0]}\sup|x_0(r,\omega_0)|^p],\mbox{ for }p\geq 2.
\end{aligned}\end{equation}

Since $(\Theta(t_0,\omega_0),p(t_0,\omega_0),q(t_0,\omega_0),P(t_0,\omega_0))\in D_{t+,x}^{1,2,1,+}V(t_0,X^*(t_0,\omega_0),X_1^*(t_0,\omega_0))$, then by Lemma \ref{lem3.2} there exists a function $\varphi\in C^{1,2,1}((0,T)\times\mathbf{R}^2)$ such that
\begin{equation}\begin{aligned}\label{eq5.9}
   &V(t_0+h,X^*(t_0+h,\omega),X_1^*(t_0+h,\omega))-V(t_0,X^*(t_0,\omega_0),X_1^*(t_0,\omega_0))\\
   &\leq\varphi(t_0+h,X^*(t_0+h,\omega),X_1^*(t_0+h,\omega))-\varphi(t_0,X^*(t_0,\omega_0),X_1^*(t_0,\omega_0)),
\end{aligned}\end{equation}
and
\begin{equation*}\begin{aligned}
  &(\Theta(t_0,\omega_0),p(t_0,\omega_0),q(t_0,\omega_0),P(t_0,\omega_0))\\
  &=(\varphi_t(t_0,X^*(t_0,\omega_0),X_1^*(t_0,\omega_0)),\varphi_x(t_0,X^*(t_0,\omega_0),X_1^*(t_0,\omega_0)),\\
  &\qquad\varphi_{x_1}(t_0,X^*(t_0,\omega_0),X_1^*(t_0,\omega_0)),\varphi_{xx}(t_0,X^*(t_0,\omega_0),X_1^*(t_0,\omega_0))).
\end{aligned}\end{equation*}
Hence we have
{\footnotesize\begin{equation}\label{eq5.10}\begin{aligned}
      &\underset{h\rightarrow 0^+}\varlimsup\mathbb{E}^{\omega_0}\bigg[\frac{V(t_0+h,X^*(t_0+h,\omega),X_1^*(t_0+h,\omega))-V(t_0,X^*(t_0,\omega),X_1^*(t_0,\omega))}{h}\bigg]\\
     =&\underset{h\rightarrow 0^+}\varlimsup\mathbb{E}^{\omega_0}\bigg[\frac{V(t_0+h,X^*(t_0+h,\omega),X_1^*(t_0+h,\omega))-V(t_0,X^*(t_0,\omega_0),X_1^*(t_0,\omega_0))}{h}\bigg]\\
  \leq&\underset{h\rightarrow 0^+}\varlimsup\mathbb{E}^{\omega_0}\bigg[\frac{\int_{t_0}^{t_0+h}\big\{\varphi_t(r,X^*(r),X_1^*(r))+
       \varphi_x(r,X^*(r),X_1^*(r))b^*(r)+\frac{1}{2}\varphi_{xx}(r,X^*(r),X_1^*(r))\sigma^*(r)\sigma^*(r)\big\}dr}{h}\bigg]\\
      &+\underset{h\rightarrow 0^+}\varlimsup\mathbb{E}^{\omega_0}\bigg[\frac{\int_{t_0}^{t_0+h}\{\varphi_{x_1}(r,X^*(r),X_1^*(r))
       [X^*(r)-\lambda X_1^*(r)-e^{-\lambda\delta }X_2^*(r)]\big\}dr}{h}\bigg]\\
     =&\underset{h\rightarrow0^+}\varlimsup\mathbb{E}^{\omega_0}\bigg[\frac{\int_{t_0}^{t_0+h}\big\{\varphi_t(r,X^{t_0,x_0;u^*}(r),X_1^{t_0,x_0;u^*}(r))
       +\varphi_x(r,X^{t_0,x_0;u^*}(r),X_1^{t_0,x_0;u^*}(r))b^*(r)\big\}dr}{h}\bigg]\\
      &+\underset{h\rightarrow 0^+}\varlimsup\mathbb{E}^{\omega_0}\bigg[\frac{\int_{t_0}^{t_0+h}\{\varphi_{x_1}(r,X^{t_0,x_0;u^*}(r),X_1^{t_0,x_0;u^*}(r))
       [X^{t_0,x_0;u^*}(r)-\lambda X_1^{t_0,x_0;u^*}(r)-e^{-\lambda\delta }X_2^{t_0,x_0;u^*}(r)]\big\}dr}{h}\bigg]\\
      &+\underset{h\rightarrow 0^+}\varlimsup \mathbb{E}^{\omega_0}\bigg[\frac{\int_{t_0}^{t_0+h}\{\frac{1}{2}\varphi_{xx}(r,X^{t_0,x_0;u^*}(r),X_1^{t_0,x_0;u^*}(r))
       \sigma^*(r)\sigma^*(r)\big\}dr}{h}\bigg],
\end{aligned}\end{equation}}
where
\begin{eqnarray*}\left\{\begin{aligned}
       b^*(r)&=b(r,X^{t_0,x_0;u^*}(r),X_1^{t_0,x_0;u^*}(r),X_2^{t_0,x_0;u^*}(r),u^*(r)),\\
  \sigma^*(r)&=\sigma(r,X^{t_0,x_0;u^*}(r),X_1^{t_0,x_0;u^*}(r),X_2^{t_0,x_0;u^*}(r),u^*(r)).
\end{aligned}\right.\end{eqnarray*}

We separately consider four items of the last equality as follows.

By Lemma \ref{lem3.2}, $\varphi,\varphi_t,\varphi_x,\varphi_{x_1},\varphi_{xx}$ has polynomial growth, thus $\varphi_t(\cdot,X^{t_0,x_0;u^*}(\cdot),X_1^{t_0,x_0;u^*}(\cdot))$, $\varphi_x(\cdot,X^{t_0,x_0;u^*}(\cdot),X_1^{t_0,x_0;u^*}(\cdot))b^*(\cdot)$, $\frac{1}{2}\varphi_{xx}(\cdot,X^{t_0,x_0;u^*}(\cdot),X_1^{t_0,x_0;u^*}(\cdot))\sigma^*(\cdot)\sigma^*(\cdot)\in L_{\mathcal{F}_{\omega_0}}^1(t_0,T;\mathbf{R})$, $\sigma^*(\cdot)\varphi_x(\cdot,X^{t_0,x_0;u^*}(\cdot),X_1^{t_0,x_0;u^*}(\cdot))$ $\in L_{\mathcal{F}_{\omega_0}}^2(t_0,T;\mathbf{R})$. Hence, by the continuity of $X^{t_0,x_0;u^*}$ and $\varphi_t$, we have
\begin{equation}\label{eq5.11}
  \frac{1}{h}\int_{t_0}^{t_0+h}\varphi_t(r,X^{t_0,x_0;u^*}(r),X_1^{t_0,x_0;u^*}(r))dr\rightarrow\varphi_t(t_0,X^{t_0,x_0;u^*}(t_0),X_1^{t_0,x_0;u^*}(t_0)),\
\end{equation}
as $h\downarrow0,\ \mathbb{P}(\cdot|\mathcal{F}_{t_0}^s)(\omega)$-$a.s.$ Consider that $\varphi_t$ has polynomial growth, and applying the dominated convergence theorem we get
\begin{equation}\begin{aligned}\label{eq5.12}
  &\underset{h\rightarrow0^+}\lim \mathbb{E}^{\omega_0}\bigg[\frac{1}{h}\int_{t_0}^{t_0+h}\varphi_t(r,X^{t_0,x_0;u^*}(r),X_1^{t_0,x_0;u^*}(r))dr\bigg]\\
  &=\mathbb{E}^{\omega_0}\big[\varphi_t(t_0,X^{t_0,x_0;u^*}(t_0),X_1^{t_0,x_0;u^*}(t_0))\big].
\end{aligned}\end{equation}

Next, we deal with the second term. We first have
\begin{equation}\label{eq5.13}\begin{aligned}
  &\mathbb{E}^{\omega_0}\bigg[\bigg|\frac{1}{h}\int_{t_0}^{t_0+h}\varphi_x(r,X^{t_0,x_0;u^*}(r),X_1^{t_0,x_0;u^*}(r))b^*(r)dr\\
  &\qquad\qquad\qquad-\varphi_x(t_0,X^{t_0,x_0;u^*}(t_0),X_1^{t_0,x_0;u^*}(t_0))b^*(t_0)\bigg|\bigg]\\
  &\leq\mathbb{E}^{\omega_0}\bigg[\frac{1}{h}\int_{t_0}^{t_0+h}\big| \varphi_x(r,X^{t_0,x_0;u^*}(r),X_1^{t_0,x_0;u^*}(r))\\
  &\qquad\qquad\qquad\quad-\varphi_x(t_0,X^{t_0,x_0;u^*}(t_0),X_1^{t_0,x_0;u^*}(t_0))\big|\big|b^*(r)\big| dr\\
  &\qquad\quad+\frac{1}{h}\int_{t_0}^{t_0+h}\big| \varphi_x(t_0,X^{t_0,x_0;u^*}(t_0),X_1^{t_0,x_0;u^*}(t_0))\big|\big|b^*(r)-b^*(t_0)\big| dr\bigg]\\
  &\leq\bigg(\frac{1}{h}\mathbb{E}^{\omega_0}\bigg[\int_{t_0}^{t_0+h}\big| \varphi_x(r,X^{t_0,x_0;u^*}(r),X_1^{t_0,x_0;u^*}(r))\\
  &\qquad\qquad\quad-\varphi_x(t_0,X^{t_0,x_0;u^*}(t_0),X_1^{t_0,x_0;u^*}(t_0))\big|^2dr\bigg]\bigg)^\frac{1}{2}
   \bigg(\frac{1}{h}\mathbb{E}^{\omega_0}\bigg[\int_{t_0}^{t_0+h}\big|b^*(r)\big|^2dr\bigg]\bigg)^\frac{1}{2}\\
  &\quad+\big|\varphi_x(t_0,X^{t_0,x_0;u^*}(t_0),X_1^{t_0,x_0;u^*}(t_0))\big|\mathbb{E}^{\omega_0}\bigg[\frac{1}{h}\int_{t_0}^{t_0+h}\big|b^*(r)-b^*(t_0)\big|dr\bigg].
\end{aligned}\end{equation}
Suppose $t_0$ is exactly chosen satisfying
\begin{equation*}
  \underset{h\rightarrow 0^+}\lim\mathbb{E}\bigg[\frac{1}{h}\int_{t_0}^{t_0+h}|b^*(r)-b^*(t_0)|dr\bigg]=0,
\end{equation*}
then we get
\begin{equation*}
  \underset{h\rightarrow 0^+}\lim\mathbb{E}\bigg[\mathbb{E}^{\omega_0}\big[\frac{1}{h}\int_{t_0}^{t_0+h}|b^*(r)-b^*(t_0)|dr\big]\bigg]=0,
\end{equation*}
hence there exists a subsequence $h_l$ such that
\begin{equation}\label{eq5.14}\begin{aligned}
\mathbb{E}^{\omega_0}\bigg[\frac{1}{h_l}\int_{t_0}^{t_0+h_l}\big|b^*(r)-b^*(t_0)\big|dr\bigg]\rightarrow 0,\ as\ h_l\downarrow 0,\ \mathbb{P}\mbox{-}a.s.\ \omega_0.
\end{aligned}\end{equation}

On the other hand, by the estimate (\ref{eq5.8}) we have
\begin{equation}\label{eq5.15}\begin{aligned}
  &\bigg(\frac{1}{h}\mathbb{E}^{\omega_0}\bigg[\int_{t_0}^{t_0+h}\big|b^*(r)\big|^2dr\bigg]\bigg)^\frac{1}{2}
   \leq C\bigg(\frac{1}{h}\mathbb{E}^{\omega_0}\bigg[\int_{t_0}^{t_0+h}\big[1+|X^{t_0,x_0;u^*}(r)|\\
  &\qquad+|X_1^{t_0,x_0;u^*}(r)|+|X_2^{t_0,x_0;u^*}(r)|\big]^2dr\bigg]\bigg)^\frac{1}{2}<+\infty.
\end{aligned}\end{equation}
Hence by the continuity of $\varphi_x$, we have
\begin{equation}\label{eq5.16}\begin{aligned}
  &\bigg(\frac{1}{h}\mathbb{E}^{\omega_0}\bigg[\int_{t_0}^{t_0+h}\big| \varphi_x(r,X^{t_0,x_0;u^*}(r),X_1^{t_0,x_0;u^*}(r))\\
  &\qquad-\varphi_x(t_0,X^{t_0,x_0;u^*}(t_0),X_1^{t_0,x_0;u^*}(t_0))\big|^2dr\bigg]\bigg)^\frac{1}{2}
   \bigg(\frac{1}{h}\mathbb{E}^{\omega_0}\bigg[\int_{t_0}^{t_0+h}\big|b^*(r)\big|^2dr\bigg]\bigg)^\frac{1}{2}\rightarrow 0,\mbox{ as }h\downarrow 0.
\end{aligned}\end{equation}
Thus, combining (\ref{eq5.14}) and (\ref{eq5.16}), we have the following estimate for the second term of (\ref{eq5.10}):
\begin{equation}\label{eq5.17}\begin{aligned}
  &\underset{h_l\rightarrow 0^+}\lim \mathbb{E}^{\omega_0}\bigg[\frac{1}{h_l}\int_{t_0}^{t_0+h_l}\varphi_x(r,X^{t_0,x_0;u^*}(r),X_1^{t_0,x_0;u^*}(r))b^*(r)dr\bigg]\\
  &=\mathbb{E}^{\omega_0}[\varphi_x(t_0,X^{t_0,x_0;u^*}(t_0),X_1^{t_0,x_0;u^*}(t_0))b^*(t_0)].
\end{aligned}\end{equation}

Similarly, employing the same argument for the last term of (\ref{eq5.10}).
Suppose $t_0$ is exactly chosen satisfying
\begin{equation*}
  \underset{h\rightarrow 0^+}\lim\mathbb{E}\bigg[\frac{1}{h}\int_{t_0}^{t_0+h}|\sigma^*(r)\sigma^*(r)-\sigma^*(t_0)\sigma^*(t_0)|dr\bigg]=0,
\end{equation*}
then there is a subsequence $h_{l'}$ so that
\begin{equation}\label{eq5.18}\begin{aligned}
 &\underset{h_{l'}\rightarrow 0^+}\lim\mathbb{E}^{\omega_0}\bigg[\frac{1}{h_{l'}}\int_{t_0}^{t_0+h_{l'}}\varphi_{xx}(r,X^{t_0,x_0;u^*}(r),X_1^{t_0,x_0;u^*}(r))\sigma^*(r)\sigma^*(r)dr\bigg]\\
&=\mathbb{E}^{\omega_0}[\varphi_{xx}(t_0,X^{t_0,x_0;u^*}(t_0),X_1^{t_0,x_0;u^*}(t_0))\sigma^*(t_0)\sigma^*(t_0)].
 \end{aligned}\end{equation}

Finally, by the continuity of $\varphi_{x_1}$, $X^{t_0,x_0;u^*}$, $X_1^{t_0,x_0;u^*}$ and $X_2^{t_0,x_0;u^*}$we derive
{\small\begin{equation}\label{eq5.19}\begin{aligned}
  &\underset{h\rightarrow 0^+}\varlimsup\mathbb{E}^{\omega_0}\bigg[\frac{\int_{t_0}^{t_0+h}\{\varphi_{x_1}(r,X^{t_0,x_0;u^*}(r),X_1^{t_0,x_0;u^*}(r))
   \big[X^{t_0,x_0;u^*}(r)-\lambda X_1^{t_0,x_0;u^*}(r)-e^{-\lambda\delta }X_2^{t_0,x_0;u^*}(r)\big]\big\}dr}{h}\bigg]\\
  &=\mathbb{E}^{\omega_0}\bigg[\varphi_{x_1}(t_0,X^{t_0,x_0;u^*}(t_0),X_1^{t_0,x_0;u^*}(t_0))\big[X^{t_0,x_0;u^*}(t_0)
   -\lambda X_1^{t_0,x_0;u^*}(t_0)-e^{-\lambda\delta }X_2^{t_0,x_0;u^*}(t_0)\big]\bigg].
\end{aligned}\end{equation}}

Summing up, we have proved that, for any sequence $h\rightarrow 0^+$, there exists a subsequence $h_{l''}$ so that
\begin{equation*}\begin{aligned}
  &\underset{h_{l''}\rightarrow 0^+}\lim\mathbb{E}^{\omega_0}\bigg[\frac{V(t_0+h_{l''},X^*(t_0+h_{l''},\omega),X_1^*(t_0+h_{l''},\omega))
   -V(t_0,X^*(t_0,\omega),X_1^*(t_0,\omega))}{h_{l''}}\bigg]\\
  &\hspace{-5mm}=\mathbb{E}^{\omega_0}\bigg[\varphi_t(t_0,X^{t_0,x_0;u^*}(t_0),X_1^{t_0,x_0;u^*}(t_0))+\varphi_x(t_0,X^{t_0,x_0;u^*}(t_0),X_1^{t_0,x_0;u^*}(t_0))b^*(t_0)\\
  &+\frac{1}{2}\varphi_{xx}(t_0,X^{t_0,x_0;u^*}(t_0),X_1^{t_0,x_0;u^*}(t_0))\sigma^*(t_0)\sigma^*(t_0)\\
  &+\varphi_{x_1}(t_0,X^{t_0,x_0;u^*}(t_0),X_1^{t_0,x_0;u^*}(t_0))\big[X^{t_0,x_0;u^*}(t_0)-\lambda X_1^{t_0,x_0;u^*}(t_0)-e^{-\lambda\delta }X_2^{t_0,x_0;u^*}(t_0)\big]\bigg].
\end{aligned}\end{equation*}
Hence by Lemma \ref{lem3.2}, we give
\begin{equation}\label{eq5.20}\begin{aligned}
    &\underset{h\rightarrow 0^+}\varlimsup\mathbb{E}^{\omega_0}\bigg[\frac{V(t_0+h,X^*(t_0+h,\omega),X_1^*(t_0+h,\omega))-V(t_0,X^*(t_0,\omega),X_1^*(t_0,\omega))}{h}\bigg]\\
\leq&\ \mathbb{E}^{\omega_0}\bigg[\varphi_t(t_0,X^{t_0,x_0;u^*}(t_0),X_1^{t_0,x_0;u^*}(t_0))+\varphi_x(t_0,X^{t_0,x_0;u^*}(t_0),X_1^{t_0,x_0;u^*}(t_0))b^*(t_0)\\
    &+\frac{1}{2}\varphi_{xx}(t_0,X^{t_0,x_0;u^*}(t_0),X_1^{t_0,x_0;u^*}(t_0))\sigma^*(t_0)\sigma^*(t_0)\\
    &+\varphi_{x_1}(t_0,X^{t_0,x_0;u^*}(t_0),X_1^{t_0,x_0;u^*}(t_0))\big[X^{t_0,x_0;u^*}(t_0)-\lambda X_1^{t_0,x_0;u^*}(t_0)-e^{-\lambda\delta }X_2^{t_0,x_0;u^*}(t_0)\big]\bigg]\\
   =&\ \Theta(t_0,\omega_0)+p(t_0,\omega_0)b^*(t_0)+\frac{1}{2}P(t_0,\omega_0)\sigma^*(t_0)\sigma^*(t_0)\\
    &\ +q(t_0,\omega_0)[X^{t_0,x_0;u^*}(t_0)-\lambda X_1^{t_0,x_0;u^*}(t_0)-e^{-\lambda\delta }X_2^{t_0,x_0;u^*}(t_0)].
\end{aligned}\end{equation}
By Fatou's lemma, we have
\begin{equation}\label{eq5.21}\begin{aligned}
    &\underset{h\rightarrow 0^+}\varlimsup \mathbb{E}\bigg[\frac{V(t_0+h,X^*(t_0+h),X_1^*(t_0+h))-V(t_0,X^*(t_0),X_1^*(t_0))}{h}\bigg]\\
\leq&\ \mathbb{E}\bigg\{\underset{h\rightarrow 0^+}\varlimsup \mathbb{E}^{\omega_0}\bigg[\frac{V(t_0+h,X^*(t_0+h),X_1^*(t_0+h))-V(t_0,X^*(t_0),X_1^*(t_0))}{h}\bigg]\bigg\}\\
\leq&\ \mathbb{E}\Big[\Theta(t_0)+ p(t_0)b^*(t_0)+\frac{1}{2}P(t_0)\sigma^*(t_0)\sigma^*(t_0)+q(t_0)[X^{t_0,x_0;u^*}(t_0)\\
    &\qquad-\lambda X_1^{t_0,x_0;u^*}(t_0)-e^{-\lambda\delta }X_2^{t_0,x_0;u^*}(t_0)]\Big].
\end{aligned}\end{equation}

Next, we want to find a function $\rho\in L^1(0,T;\mathbf{R})$ and some $h_0>0$ such that
\begin{equation}\label{eq5.22}
  \frac{\mathbb{E}\big[V(t+h,X^*(t+h),X_1^*(t+h))-V(t,X^*(t),X_1^*(t))\big]}{h}\leq\rho(t),\ a.e.\ t\in[0,T],\ h\leq h_0.
\end{equation}
Choose $h_0$ enough small, we have
{\footnotesize\begin{equation*}\begin{aligned}
    &\frac{\mathbb{E}[V(t_0+h,X^*(t_0+h),X_1^*(t_0+h))-V(t_0,X^*(t_0),X_1^*(t_0))]}{h}\\
   =&\ \frac{\mathbb{E}\big[\mathbb{E}^{\omega_0}[V(t_0+h,X^*(t_0+h),X_1^*(t_0+h))-V(t_0,X^*(t_0),X_1^*(t_0))]\big]}{h}\\
\leq&\ \frac{\mathbb{E}\big[\mathbb{E}^{\omega_0}[\varphi(t_0+h,X^*(t_0+h),X_1^*(t_0+h))-\varphi(t_0,X^*(t_0),X_1^*(t_0))]\big]}{h}\\
   =&\ \frac{\mathbb{E}\big[\mathbb{E}^{\omega_0}[\int_{t_0}^{t_0+h}\big\{\varphi_t(r,X^{t_0,x_0;u^*}(r),X_1^{t_0,x_0;u^*}(r))
     +\varphi_x(r,X^{t_0,x_0;u^*}(r),X_1^{t_0,x_0;u^*}(r))b^*(r)\big\}dr]\big]}{h}\\
    &\ +\frac{\mathbb{E}\big[\mathbb{E}^{\omega_0}[\int_{t_0}^{t_0+h}\{\varphi_{x_1}(r,X^{t_0,x_0;u^*}(r),X_1^{t_0,x_0;u^*}(r))[X^{t_0,x_0;u^*}(r)
     -\lambda X_1^{t_0,x_0;u^*}(r)-e^{-\lambda\delta }X_2^{t_0,x_0;u^*}(r)]\big\}dr]\big]}{h}\\
    &\ +\frac{\mathbb{E}\big[\mathbb{E}^{\omega_0}[\int_{t_0}^{t_0+h}\{\frac{1}{2}\varphi_{xx}(r,X^{t_0,x_0;u^*}(r),X_1^{t_0,x_0;u^*}(r))\sigma^*(r)\sigma^*(r)\big\}dr]\big]}{h}\\
\leq&\ \frac{C}{h}\mathbb{E}\bigg\{\mathbb{E}^{\omega_0}\bigg[\int_{t_0}^{t_0+h}\big[1+|X^{t_0,x_0;u^*}(r)|^k+|X_1^{t_0,x_0;u^*}(r)|^k\big]\big[1+|X^{t_0,x_0;u^*}(r)|^2\\
    &\qquad+|X_1^{t_0,x_0;u^*}(r)|^2+|X_2^{t_0,x_0;u^*}(r)|^2\big]dr\bigg]\bigg\}\\
\end{aligned}\end{equation*}}
{\small\begin{equation*}\begin{aligned}
\leq&\ \frac{C}{h}\mathbb{E}\bigg\{\mathbb{E}^{\omega_0}\bigg[\int_{t_0}^{t_0+h}\big[1+|X^{t_0,x_0;u^*}(r)|^{4+2k}+|X_1^{t_0,x_0;u^*}(r)|^{4+2k}
     +|X_2^{t_0,x_0;u^*}(r)|^4\big]dr\bigg]\bigg\}\\
\leq&\ C\mathbb{E}\bigg\{\mathbb{E}^{\omega_0}\bigg[1+\underset{t_0\leq r\leq T}\sup|X^{t_0,x_0;u^*}(r)|^{4+2k}+\underset{t_0-\delta\leq r\leq t_0}\sup|x_0(r)|^{4+2k}\bigg]\bigg\}\\
\leq&\ C\mathbb{E}\bigg[1+\underset{t_0-\delta\leq r\leq t_0}\sup|x_0(r)|^{4+2k}\bigg]\leq C\mathbb{E}\bigg[1+\underset{-\delta\leq r\leq 0}\sup|\varphi(r)|^{4+2k}\bigg].
\end{aligned}\end{equation*}}
Choose $\rho(t_0)=C\mathbb{E}\Big[1+\underset{-\delta\leq r\leq 0}\sup|\varphi(r)|^{4+2k}\Big]$, since $t_0$ is of full measure in $[0,T]$, (\ref{eq5.22}) is proved.

Now, we apply Lemma \ref{lem5.1} to the function $\lambda(t)=\mathbb{E}\big[V(t,X^*(t),X_1^*(t))\big]$, using Fatou's lemma and (\ref{eq5.21}) to get
\begin{equation}\label{eq5.23}\begin{aligned}
  &\mathbb{E}[V(T,X^*(T),X_1^*(T))]-V(s,x,x_1)\\
  &\leq\int_s^T\underset{h\rightarrow 0^+}\varlimsup\mathbb{E}\bigg[\frac{V(t+h,X^*(t+h),X_1^*(t+h))-V(t,X^*(t),X_1^*(t))}{h}\bigg]dt\\
  &\leq\int_s^T\mathbb{E}\bigg[\Theta(t)+ p(t)b^*(t)+\frac{1}{2}P(t)\sigma^*(t)\sigma^*(t)+q(t)[X^*(t)-\lambda X_1^*(t)-e^{-\lambda\delta }X_2^*(t)]\bigg]dt.
\end{aligned}\end{equation}
Consequently, by (\ref{eq5.6}) we have
\begin{equation}\label{eq5.24}\begin{aligned}
   &\mathbb{E}[V(T,X^*(T),X_1^*(T))]-V(s,x,x_1)\\
   &\leq \mathbb{E}\int_s^T\big[a(t,X^*(t),X_1^*(t),X_2^*(t),u^*(t))-f(t)V(t,X^*(t),X_1^*(t))\big]dt\\
   &\leq \mathbb{E}\int_s^T\big[a(t,X^*(t),X_1^*(t),X_2^*(t),u^*(t))+f(t)Y^*(t)\\
   &\qquad\qquad+C|Y^*(t)+V(t,X^*(t),X_1^*(t))|\big]dt.
\end{aligned}\end{equation}
Noting $Y^*(s)$ is deterministic, thus
\begin{equation*}
  Y^*(s)=\mathbb{E}\bigg[-V(T,X^*(T),X_1^*(T))+\int_s^T\big[a(t,X^*(t),X_1^*(t),X_2^*(t),u^*(t))+f(t)Y^*(t)\big]dt\bigg],
\end{equation*}
hence we obtain
\begin{equation}\label{eq5.25}\begin{aligned}
      &V(s,x,x_1)\\
  \geq&\ \mathbb{E}\bigg[V(T,X^*(T),X_1^*(T))-\int_s^T\big[a(t,X^*(t),X_1^*(t),X_2^*(t),u^*(t))+f(t)Y^*(t)\big]dt\\
      &\quad-C\int_s^T|Y^*(t)+V(t,X^*(t),X_1^*(t))|dt\bigg]\\
  =&-Y^*(s)-C\mathbb{E}\bigg[\int_s^T|Y^*(t)+V(t,X^*(t),X_1^*(t))|dt\bigg]\\
  =&-Y^*(s)+C\mathbb{E}\bigg[\int_s^T(Y^*(t)+V(t,X^*(t),X_1^*(t)))dt\bigg].
\end{aligned}\end{equation}
Finally, by the backward Gronwall's inequality we get
\begin{equation}\label{eq5.26}
  V(s,x,x_1)\geq-Y^*(s)=J(s,\varphi;u^*(\cdot)).
\end{equation}
Combining (\ref{eq5.26}) with (\ref{eq5.4}) shows that $u^*(\cdot)$ is an optimal control.
\end{proof}

\begin{Corollary} Suppose in Theorem \ref{thm5.1}, $g(\cdot)$ is any given uniformly bounded deterministic function, then the corresponding HJB equation becomes
(\ref{prop 5.1}). In this case we can still obtain the same conclusion as Theorem \ref{thm5.1}.
\end{Corollary}
\begin{proof}
Apply Girsanov's theorem, we can rewrite (\ref{eq5.1}) as
\begin{eqnarray}\label{eq5.27}\left\{\begin{aligned}
  dX^{s,\varphi;u}(t)=&\big[b(t,X^{s,\varphi;u}(t),X_1^{s,\varphi;u}(t),X_2^{s,\varphi;u}(t),u(t))\\
                      &\ +\sigma(t,X^{s,\varphi;u}(t),X_1^{s,\varphi;u}(t),X_2^{s,\varphi;u}(t),u(t))g(t)\big]dt\\
                      &+\sigma(t,X^{s,\varphi;u}(t),X_1^{s,\varphi;u}(t),X_2^{s,\varphi;u}(t),u(t))d\tilde{W}(t),\ t\in[s,T],\\
 -dY^{s,\varphi;u}(t)=&\big[a(t,X^{s,\varphi;u}(t),X_1^{s,\varphi;u}(t),X_2^{s,\varphi;u}(t),u(t))+f(t)Y^{s,\varphi;u}(t)\big]dt\\
                      &-Z^{s,\varphi;u}(t)d\tilde{W}(t),\ t\in[s,T],\\
   X^{s,\varphi;u}(t)=&\ \varphi(t-s),\ t\in[s-\delta,s],\\
   Y^{s,\varphi;u}(T)=&\ \phi(X^{s,\varphi;u}(T),X_1^{s,\varphi;u}(T)),
\end{aligned}\right.\end{eqnarray}
where $\tilde{W}(t):=W(t)-\int_s^tg(r)dr$. We define a new probability measure $\mathbb{Q}$ on $(\Omega,\mathcal{F}_t^s)$ by
\begin{equation*}
  \frac{d\mathbb{Q}}{d\mathbb{P}}\bigg|_{\mathcal{F}_t^s}=\exp\bigg\{\int_s^tg(r)dW(r)-\frac{1}{2}\int_s^tg^2(r)dr\bigg\}.
\end{equation*}
Apparently in the new probability measure space $(\Omega,\mathcal{F},\mathbb{Q})$, $\tilde{W}(\cdot)$ is a standard Brownian motion. Furthermore, the cost functional can be rewritten as
\begin{equation}\label{eq5.28}\begin{aligned}
  &J(s,\varphi;u(\cdot))=-Y^{s,\varphi;u}(s)|_{(\Omega,\mathcal{F},\mathbb{Q})}\\
  &=-\mathbb{E}_\mathbb{Q}\bigg[\int_s^Te^{\int_s^tf(r)dr}a(t,X^{s,\varphi;u}(t),X_1^{s,\varphi;u}(t),X_2^{s,\varphi;u}(t),u(t))dt\\
  &\qquad\qquad+e^{\int_s^Tf(r)dr}\phi(X^{s,\varphi;u}(T),X_1^{s,\varphi;u}(T))\bigg],
\end{aligned}\end{equation}
where $\mathbb{E}_\mathbb{Q}[\cdot]$ denotes the expectation under the probability measure $\mathbb{Q}$. Under the new probability space $(\Omega,\mathcal{F},\mathbb{Q})$, \textbf{Problem (P)} can be restated as the following.

\textbf{Problem (Q).} For given $(s,\varphi)\in[0,T]\times C([-\delta,0];\mathbf{R})$, the object is to find $u^*(\cdot)\in\mathcal{U}^\omega[s,T]$ such that (\ref{eq5.27}) is satisfied and (\ref{eq5.28}) is minimized.

Finally, by Theorem \ref{thm5.1}, the proof of the corollary is completed.
\end{proof}

\section{Concluding Remarks}

This paper is an extension of Shi et al. \cite{SXZ15}. With the help of the viscosity solution, we establish a nonsmooth version of the connection between the adjoint variables and the value function for stochastic recursive optimal control problem with mixed delay. The connection can be interpreted as two inclusions, the first one is between $\{p_1(t)(\gamma(t))^{-1}\}$ and $D^{1,+}_{x}V(t,X^*(t),X_1^*(t))$, the second one is between $\{p_1(t)(\gamma(t))^{-1}\}$ and $D^{1,-}_{x}V(t,X^*(t),X_1^*(t))$, which is consistent with the early results in Theorem 3.2 of \cite{SXZ15} when the value function is smooth. The connection between the adjoint variables and the value function can help look for the optimal control, thus we also give the verification theorem to verify if an admissible control is really optimal.

An interesting and challenging problem is to research the corresponding relationship with nonconvex domain, as Nie et al. \cite{NSW17}. In this case, a global maximum principle to generalize the result by Hu \cite{Hu17} to the time delayed case is necessary. We will consider this topic in the future.


\end{document}